\documentclass[12pt,english]{article}
\usepackage[T1]{fontenc}
\usepackage{geometry}
\usepackage{amsmath}
\usepackage{amssymb}
\usepackage{amsthm}
\usepackage{color}
\usepackage{graphicx}
\usepackage{epstopdf}
\usepackage{subfigure}
\usepackage{float}
\usepackage{geometry}
\usepackage{authblk}

\usepackage[noend]{algpseudocode}
\usepackage{algorithmicx,algorithm}
\makeatletter
\@ifundefined{date}{}{\date{}}

\newcommand{\rom}[1]{\uppercase\expandafter{\romannumeral#1}}

\theoremstyle{plain}\newtheorem{theorem}{Theorem}[section]
\theoremstyle{definition}\theoremstyle{plain}\newtheorem{lemma}[theorem]{Lemma}\newtheorem{corollary}{Corollary}\newtheorem{remark}{Remark}

\newtheorem{example}{Example}[section]

\usepackage{epstopdf}
\usepackage{subfigure}

\usepackage{physics}

\graphicspath{{fig/}}

\newcommand{\bm}{\boldsymbol}

\makeatother

\usepackage{babel}

\begin{document}

\title{Quantum simulation of a class of
highly-oscillatory transport equations via Schr{\"o}dingerisation}

\author[1,2]{Anjiao Gu}

\author[1,2,3]{Shi Jin \thanks{Corresponding author: shijin-m@sjtu.edu.cn.}}

\affil[1]{School of Mathematical Sciences, Shanghai Jiao Tong University, Shanghai 200240, China}

\affil[2]{Institute of Natural Sciences, Shanghai Jiao Tong University, Shanghai 200240, China}

\affil[3]{Ministry of Education Key Laboratory in Scientific and Engineering Computing,  Shanghai 200240, China}

\renewcommand*{\Affilfont}{\small\it}
\renewcommand\Authands{ and }

\date{}
\maketitle

\abstract{
In this paper, we present quantum algorithms for a class of highly-oscillatory transport equations, which arise in semi-classical computation  of surface hopping problems and other related non-adiabatic quantum dynamics, based on the Born-Oppenheimer approximation. Our method relies on the classical nonlinear geometric optics method, and the recently developed Schr{\"o}dingerisation approach for quantum simulation of partial differential equations. 
The Schr{\"o}dingerisation technique can transform any linear ordinary and partial differential equations into Hamiltonian systems evolving under unitary dynamics, via a warped phase transformation that maps these equations to one higher dimension.
We study possible paths for better recoveries of the solution to the original problem by shifting the bad eigenvalues in the Schr\"odingerized system.  
Our method ensures the uniform error estimates {\it independent} of the wave length, thus allowing numerical accuracy, in maximum norm, even without numerically resolving the physical oscillations. Various numerical experiments are performed to demonstrate the validity of this approach.
}

{\bf{Keywords:}} Quantum simulation, Schr{\"o}dingerisation, highly oscillatory transport PDEs, nonlinear geometric optics method

\section{Introduction}

As the classical computer is increasingly closer to its physical limit~\cite{Freiser1969}, a potentially effective path to the future computing platform is quantum computing~\cite{Feynman1982} 
which are currently under rapid development. 
Compared with classical methods, quantum algorithms have already been shown to have certain advantages in a number of problems~\cite{Deutsch1985quantum,Ekert1998quantum,Nielsen2010}. 
For instance, the HHL algorithm proposed by Harrow, Hassidim and Lloyd in \cite{Harrow2009Quantum} can realize exponential acceleration when solving systems of linear equations under several conditions.
Such successful cases motivate one to extend the high-cost classical algorithms of important scientific problems to quantum algorithms~\cite{Shor1994,DiVincenzo1995quantum,Andrew1998}.
In particular, quantum algorithms for solving ordinary or partial differential equations have received extensive recent attention~\cite{Berry2014Highorder,Montanaro2016FEM,Engel2019Vlasov,Costa2019wave,Childs2020Spectral,Childs2021highprecision,Jin2022quantum}.

In this paper, we concentrate on a family of transport equations in which the phase oscillations depend on both time and space as follows
\begin{align*}
&\partial_t u+\sum_{k=1}^d A_k(x)\partial_{x_k} u=\frac{iE(t,x)}{\varepsilon}Du+Cu,\\
&u(0,x)=u_0(x),\ x\in\Omega\subset\mathbb{R}^d,
\end{align*}
where $u(t,x)\in\mathbb{C}^n$, $A_k$ and $D$ are $n\times n$ real symmetric matrices, $C\in\mathbb{R}^{n\times n}$, $E$ is a real scalar function. The initial function is the given initial data which may have a periodic oscillation with the phase $\beta(x)/\varepsilon$, i.e. $u_0(x)=f_0(x,\beta(x)/\varepsilon)$.
A number of semi-classical models in quantum dynamics can be written in this form, such as surface hopping~\cite{Chai2015}, graphene~\cite{Morandi2011Wigner}, quantum dynamics in periodic lattice~\cite{Morandi2009Multiband}, etc. 
Solving such systems is computationally daunting since 1) often the dimension is high ($d\gg1$), and 2) one needs to numerically resolve the small wave lengths which are of $O(\varepsilon)$~\cite{Maslov2001Semiclassical,Fatemi1995}. While Hamiltonian simulation provides a principled approach to solving the Schr{\"o}dinger formulation of such applications-for instance, via quantum algorithms designed for semi-classical Schr{\"o}dinger equations~\cite{Jin2022semiclassical,Bornsweil2025Semiclassical}, these methods do  need to use mesh sizes that {\it depend on} $\varepsilon$.

The nonlinear geometric optics (NGO) based method, developed in \cite{Nicolas2017Nonlinear} for above problems, uses a nonlinear eometric optics ansatz which builds the oscillatory phase as an independent variable, and a suitably chosen initial data derived from the Chapman-Enskog expansion,  guarantee that one can use mesh size and time step {\it independent of }$\varepsilon$. However, the curse-of-dimensionality remains the bottleneck for classical computation. This motivates us to consider quantum algorithms for this problem, which has no curse-of-dimensionality due to the use of qubits, since in general one just needs $n=\mathcal{O}(d\log N)$ qubits for $d$-dimensional problems, where $N$ is the number of discretization points per dimension.

Although existing Hamiltonian simulation algorithms (e.g., quantum Magnus expansion methods~\cite{Casares2024Magnus,Sharma2024Magnus,Fang2025Magnus,Bornsweil2025Magnus}) can effectively simulate many highly oscillatory systems governed by anti-Hermitian operators, the core difficulty in this case arises from the non-anti-Hermitian character of the operator, primarily due to the inclusion of the $C$ term.
Our main tool for quantum simulation is a new method called {\it Schr{\"o}dingerisation} which have been recently proposed for solving general linear ordinary and partial differential equations~\cite{Jin2022quantum,Jin2023Quantum}.
This technique can convert non-unitary dynamics (linear differential equations) to unitary dynamics (Schr{\"o}dinger type equations) by a warped phase transformation that just maps the system to one higher dimension, followed by a Fourier transform.
The original variable can be recovered via integration on or pointwise evaluation of the extra variable.
This method works for {\it both} qubits~\cite{Jin2022quantum,Jin2023Quantum} {\it and} continuous-variable frameworks~\cite{Jin2023Analog}, open quantum systems with artificial boundary conditions~\cite{Jin2023artificialboundary}, physical boundary or interface conditions~\cite{Jin2024physicalboundary}, etc. 
and problems with time-dependent coefficients \cite{Berry2020timedependent,An2021timedependent,An2022timedependent,Fang2023timemarchingbased, Cao2023quantum}.
Other possible candidates for the problem include discretization based approaches such as linear algebra-based  HHL algorithm \cite{Harrow2009Quantum}, and the linear combination of Hamiltonian simulation, as outlined in~\cite{An2023LCHS,An2023Quantum,An2024Laplace}, if one is interested in qubit framework, not  continuous variable frameworks.

The outline of this paper is as follows. 
In Section 2, we give a brief review of the Schr{\"o}dingerisation method. 
Then, based on the nonlinear geometric optics approach, we apply the Schr{\"o}dingerisation algorithm together with spectral and finite difference discretisations on the equations, for one dimensional scalar equations in Section 3.
We provide  numerical results for problems of different settings.
In Section 4, we extend the new approach to a class of PDE system which is the semi-classical approximation of surface hopping problem \cite{Chai2015}, and show the effectiveness of our method for such highly-oscillatory transport problems by presenting some typical numerical examples.
Finally, we conclude this paper in Section 5. 

Throughout this paper, we consider the problems in a finite time interval, i.e. $T=\mathcal{O}(1)$. 
For arbitrary vector $\boldsymbol{a}=(a_0,a_1,\cdots,a_{n-1})^\top$, we use the notation ${\text {diag}}\left\{\boldsymbol{a}\right\}={\text{diag}}\left[a_0,a_1,\cdots,a_{n-1}\right]$ to denote the matrix $\sum_j a_j\ket{j}\bra{j}$.
Unless otherwise specified, we adopt the uniform grids for the independent variables. 
Moreover, for a given dimension, such as $x$, we denote the number of grid points by $M_x=2^{m_x}$, and similarly for the other dimensions.

\section{Description of the Schr{\"o}dingerisation}\label{review}

We begin with a brief presentation of the Schr{\"o}dingerisation method.
It was first proposed in~\cite{Jin2022quantum,Jin2023Quantum} and can be applied to quantum simulation for general linear dynamical systems, including both systems of ODEs and PDEs:
\begin{equation}\label{ODEs}
\frac{d\boldsymbol{u}}{dt}=A\boldsymbol{u}+\boldsymbol{b},\quad \boldsymbol{u}(0)=\boldsymbol{u}_0,
\end{equation}
where $\boldsymbol{u},\boldsymbol{b}\in\mathbb{C}^n$ and $A\in\mathbb{C}^{n\times n}$ is a linear (for ODEs) or linear differential (for PDEs) operator. 
We firstly present the Schr{\"o}dingerisation method for time-independent $A$ and $b$. 
First, \eqref{ODEs} can be rewritten as a homogeneous system:
\begin{equation}\label{homoeq}
\frac{d\tilde{\boldsymbol{u}}}{dt}=\tilde{A}\tilde{\boldsymbol{u}},\quad
\tilde{\boldsymbol{u}}=\left[\begin{array}{c}
    \boldsymbol{u} \\
    \boldsymbol{r}
\end{array}\right],\quad
\tilde{A}=\left[\begin{array}{cc}
    A & \frac{I}{T} \\
    \boldsymbol{0} & \boldsymbol{0}
\end{array}\right],\quad
\tilde{\boldsymbol{u}}(0)=\left[\begin{array}{c}
    \boldsymbol{u}_0 \\
    T\boldsymbol{b}
\end{array}\right],
\end{equation}
where $T$ is the final time.
 $\tilde{A}$ can be further decomposed into a Hermitian term and an anti-Hermitian one:
\begin{equation}\label{Hmat}
\tilde{A}=H_1+iH_2,\quad H_1=\frac{\tilde{A}+\tilde{A}^\dagger}{2},\quad H_2=\frac{\tilde{A}-\tilde{A}^\dagger}{2i}.
\end{equation} 

Using the warped phase transformation $\boldsymbol{v}(t,p)=e^{-p}\tilde{\boldsymbol{u}}(t)$ for $p>0$ and properly extending the initial data to $p<0$, equation \eqref{homoeq} are then transferred to linear convection equations:
\begin{equation}\label{linearconv}
\frac{d\boldsymbol{v}}{dt}=-H_1\partial_p\boldsymbol{v}+iH_2\boldsymbol{v},\quad \boldsymbol{v}(0)=\xi(p)\tilde{\boldsymbol{u}}(0).
\end{equation}
Here, the smoothness of $\xi(p)$ has an impact on the convergence rate of the numerical method.
For instance, if we use $\xi(p)=e^{-\vert p \vert}$, it implies a first-order accuracy on the spatial discretization due to the regularity in $p$ of the initial condition.
If one uses a smoother initial value of $\boldsymbol{v}(0)$ with
\begin{equation}\label{xip}
\xi(p)=
\begin{cases}
(-3+3e^{-1})p^3+(-5+4e^{-1})p^2-p+1,& p\in(-1,0),\\
e^{-\vert p\vert},& \textit{otherwise},
\end{cases}
\end{equation}
it gives a second-order accuracy~\cite{Jin2024schrodingerization, Jin2024illposed}. Higher order accuracy can also be achieved by requiring smoother $\xi(p)$, see \cite{Jin2024illposed}.

\subsection{The discrete Fourier transform for Schr{\"o}dingerisation}\label{sec_DFT}

We discuss the discrete Fourier transformation (DFT) on $p$ in this section.
We use the uniform mesh size $\Delta{p}=(R-L)/M$ for $M=2N$ and denote the grid points as
$$p_j=L+j\Delta{p},\quad j=0,1,\cdots,M.$$
The basis functions are chosen as
$$\phi_l(x)=e^{i\mu_l(x-L)},\quad \mu_l=\frac{2\pi (l-N)}{R-L},\quad l\in[M].$$
Here $[M]=\{0,1,\cdots,M-1\}$.
Then we can define
$$\Phi_p=(\phi_{jl})_{M\times M}=(\phi_l(p_j))_{M\times M},\quad D_p={\text{diag}}\{\mu_0,\cdots,\mu_{M-1}\}.$$
After the transformation on $p$, equation \eqref{linearconv} is transferred to
\begin{equation*}
i\frac{d\boldsymbol{w}}{dt}=(H_1\otimes P_p-H_2\otimes I_{M_p})\boldsymbol{w},
\end{equation*}
where $\boldsymbol{w}=\sum_{j,k} v_j(t,p_k)\ket{j}\ket{k}$ and $P_p=\Phi_p D_p (\Phi_p)^{-1}=P_p^{\dagger}$.
By a change of variables $\tilde{\boldsymbol{w}}=(I_{M_x}\otimes(\Phi_p)^{-1})\boldsymbol{w}$, one can get a "Schr\"odingerized"  system:
\begin{equation}\label{SchODE}
i\frac{d\tilde{\boldsymbol{w}}}{dt}=(H_1\otimes D_p-H_2\otimes I_{M_p})\tilde{\boldsymbol{w}}=:H\tilde{\boldsymbol{w}}.
\end{equation}
In practical calculation, the Hamiltonian $H$ is usually sparse, and has the following properties
$$s(H)=\mathcal{O}(s(\tilde{A})),\quad \Vert H\Vert_{max}\le \Vert H_1\Vert_{max}/\Delta{p}+\Vert H_2\Vert_{max},$$
where $s(H)$ is the sparsity (maximum number of nonzero entries in each row and column) of the matrix $H$ and $\Vert H\Vert_{max}$ is its max-norm  (value of largest entry in absolute value). 
For general sparse Hamiltonian simulation, an optimal quantum algorithm can be found in~\cite{Low2017PRL}, with complexity given by the next lemma.
\begin{lemma}\label{LemSch}

An $s$-sparse Hamiltonian $H$ on $m_{H}$ qubits with matrix elements specified to $n$ bits of precision can be simulated for time interval $T$, error $\epsilon$, and success probability at least $1-2\epsilon$ with 
\begin{equation*}
\mathcal{O}\left(sT\Vert H\Vert_{max}+\frac{\log(1/\epsilon)}{\log\log(1/\epsilon)}\right)
\end{equation*}
queries and and a factor
\begin{equation*}
\mathcal{O}\left(m_{H}+n\mathrm{polylog}(n)\right)
\end{equation*}
additional gates.

\end{lemma}

\subsection{Recovery of the solution}\label{retrieval}

We denote $\lambda(H_1)$ the set of eigenvalues of $H_1$. To discretize the $p$ domain, we need to choose a large enough domain $p\in[L,R]$ such that  
\begin{equation}\label{eq:require_L}
e^{L+2\lambda_{\max}^+(H_1)T+R_p}\leq\epsilon,\quad 
e^{L+\lambda_{\max}^-(H_1)T+\lambda_{\max}^+(H_1)T+R_p}\leq\epsilon,
\end{equation}
where $\epsilon$ is the desired accuracy, $R_p\geq 1$ is the length of the recovery region with $e^{R_p} = \mathcal{O}(1)$, and 
\begin{equation*}
\begin{aligned}
\lambda_{\max}^{+}(H_1):= \max\bigg\{\sup_{0<t<T} \{|\lambda|: \lambda\in\lambda(H_1(t)), \lambda>0\},0\bigg\},\\
\lambda_{\max}^{-}(H_1):= \max\bigg\{\sup_{0<t<T}\{|\lambda|: \lambda\in\lambda(H_1(t)), \lambda<0\},0\bigg\}.
\end{aligned}
\end{equation*}

If all eigenvalues of $H_1$ are negative, the convection term of \eqref{linearconv} corresponds to a wave moving from the right to the left, thus it does not need to impose a boundary condition for $\boldsymbol{v}$ at $p=0$. 
For the general case, i.e. if $H_1$ has non-negative eigenvalues, one uses the following Theorem to recover the original solution.
\begin{theorem}\label{Thforp}
\cite{Jin2024schrodingerization}
The solution $\tilde{\boldsymbol{u}}(T)$ can be restored by
\begin{equation*}
\tilde{\boldsymbol{u}}=e^{p_k}\boldsymbol{v}(p_k),\ \textit{or}\ \tilde{\boldsymbol{u}}=e^{p_k}\int_{p_k}^{\infty}\boldsymbol{v}dp,
\end{equation*}
where $p_k\ge p_\star=\lambda_{\max}^{+}(H_1)T$.

\end{theorem}

The error estimates for the discretizations of the Schr{\"o}dingerized system are given by the following Theorem.
\begin{theorem}\label{errThforSch}
\cite{Jin2024schrodingerization}
Suppose $L=-R$ satisfies \eqref{eq:require_L} with $\epsilon$ the desired accuracy. 
Assume $\tilde{\bm{u}}_h$ is the approximation of $\tilde{\bm{u}}$ recovered by $p_k$ in Theorem~\ref{Thforp}, and the initial value of $\bm{v}(0)$ as $\bm{v}(0)=\xi(p)\tilde{\bm{u}}(0)$, where $\xi(p)\in H^r(\mathbb{R})$ satisfies $\xi(p)=e^{-p} $ for $p\geq 0$.
There holds 
\begin{equation*}
\|\tilde{\bm{u}}_h(T)-\tilde{\bm{u}}(T)\|_{L^2(\Omega_p)}\lesssim \Delta p^r e^{p_k} \|\tilde{\bm{u}}(T)\|+\epsilon\|\tilde{\bm{u}}_0\|,
\end{equation*}
where $\Omega_p=(p_k,p_k+R_p)$,
$R_p\geq 1$ is the length of the recovery region with $e^R_p =\mathcal{O}(1)$.
   
\end{theorem}

Assume the eigenvalues of $H_1$ in \eqref{Hmat} satisfy
$$\lambda_1(H_1)\le\lambda_2(H_1)\le\cdots\le\lambda_n(H_1),$$ 
the above Theorem means that we have to use qubits $\ket{p_k}$ satisfying $p_k\ge\lambda_n(H_1)T$ to recover the original variables for systems that contain unstable modes.
On the other hand, we can further give the following corollary.
\begin{corollary}\label{lambdachoice}

For the homogeneous system \eqref{homoeq}, we introduce the following initial-value problem:
\begin{equation}\label{newu}
\frac{d\tilde{\boldsymbol{u}}^{\lambda}}{dt}=(H_1-\lambda_0 I_{2n}+iH_2)\tilde{\boldsymbol{u}}^{\lambda},\ \tilde{\boldsymbol{u}}^{\lambda}(0)=\tilde{\boldsymbol{u}}(0).
\end{equation}
where $\lambda_0\in\mathbb{R}$.
Then, by denoting $\boldsymbol{v}^{\lambda}=\xi(p)\tilde{\boldsymbol{u}}^{\lambda}$ we can restore the solution $\tilde{\boldsymbol{u}}^{\lambda}$ by
\begin{equation}\label{restore}
\tilde{\boldsymbol{u}}^{\lambda}=e^{p_k}\boldsymbol{v}^{\lambda}(p_k),\ \textit{or}\ \tilde{\boldsymbol{u}}^{\lambda}=e^{p_k}\int_{p_k}^{\infty}\boldsymbol{v}^{\lambda}dp
\end{equation}
where $p_k\ge p_\star=\max\{(\lambda_n(H_1)-\lambda_0)T,0\}$.
And $\tilde{\boldsymbol{u}}(T)$ can be recovered by $e^{\lambda_0 T}\tilde{\boldsymbol{u}}^{\lambda}(T)$.

\end{corollary}

The above corollary gives us the freedom to shift the eigenvalues of $H_1$ for changing the $p_{\star}$ in Theorem \ref{Thforp}.
For instance, the new Hermite part $H_1-\lambda_0 I_{2n}$ only has non-positive eigenvalues when $\lambda_0\ge\lambda_n(H_1)$.
In this situation, it has $p_\star=0$.
On the other hand, it can be observed that the choice of $\lambda_0$ will influence the region of $p$.
Thus, one can also adopt a $\lambda_0$ which is close to $\frac{\lambda_1(H_1)+\lambda_n(H_1)}{2}$ to release $L$ in \eqref{eq:require_L}.
It's easy to check that changing the eigenvalues may not always brings good results according to Theorem~\ref{errThforSch}. 
However, we still can select an appropriate $\lambda_0$ based on the actual situation to seek a better recovery.
The followings are three illustrations.
\begin{example}
Let $\delta_T$ be defined as $\Vert\tilde{\bm{u}}(0)/\tilde{\bm{u}}(T)\Vert$. Then the following conclusions hold:
\begin{itemize}
	\item If $\lambda_n(H_1)> 0$ and $\Delta p^r e^{\lambda_n(H_1)T}>\epsilon\delta_T$, we can use $\lambda_0\in[0,\lambda_n(H_1)]$ to get a smaller $p_k$ for recovery while simultaneously reducing the error.
	\item If $\lambda_n(H_1)> 0$ and $\Delta p^r e^{\lambda_n(H_1)T}\le\epsilon\delta_T$, we can use $\lambda_0\in[0,\lambda_n(H_1)]$ to get a smaller $p_k$ for recovery without increasing the error.
	\item If $\lambda_n(H_1)<0$ and $\Delta p^r<\epsilon\delta_T$, we can choose $\lambda_0\in[\lambda_b,0]$ for achieving more accurate solutions, where $\lambda_b:=\max\{\lambda_n(H_1),\log\left(\frac{\Delta{p}^r}{\epsilon\delta_T}\right)/T\}$.
\end{itemize}

\end{example}

For the quantum state, one can directly employ a measurement on the $p$-register.
After performing Hamiltonian simulation to implement the unitary operator $e^{-iHT}$ on initial state $\ket{\tilde{\psi}(0)}=\frac{\tilde{\boldsymbol{w}}(0)}{\Vert\tilde{\boldsymbol{w}}(0)\Vert}$, we will get the state $\ket{\tilde{\psi}(T)}$.
One can first apply the quantum Fourier transform (QFT) on the resulting state to retrieve $\ket{\psi(T)}$.
By choosing an appropriate $p_k\ge p_\star$ to apply quantum measurement $I_{M_x}\otimes\ket{p_k}\bra{p_k}$ on $\ket{\psi(T)}$, then the state are collapsed to
$$\ket{\psi_k(T)}\approx\frac{1}{\mathcal{N}}\left(\sum\limits_{j} v_j(T,p_k)\ket{j}\right)\otimes\ket{k},\ \mathcal{N}=\left(\sum\limits_{j} \vert v_j(T,p_k)\vert^2\right)^{1/2}$$
with probability $P_r(T,p_k)=\frac{\Vert\boldsymbol{v}(T,p_k)\Vert^2}{\sum_{l}\Vert \boldsymbol{v}(T,p_l)\Vert^2}$.
Moreover, a direct computation gives
$$P_r(T,p\ge p_\star)=\frac{1}{2}\frac{\Vert e^{-p_\star}\tilde{\boldsymbol{u}}(T)\Vert^2}{\Vert \tilde{\boldsymbol{u}}(0)\Vert^2}+\mathcal{O}(\Delta p).$$

\subsection{Application to non-autonomous ordinary differential equations}\label{timedepway}

For the time-dependent Hamiltonian $H$ in \eqref{SchODE}, we can adopt the method proposed in~\cite{Cao2023quantum}.
By adding a new variable $s\in\mathbb{R}$, \eqref{SchODE} is then transformed to an autonomous equation
\begin{equation*}
\frac{\partial\boldsymbol{z}}{\partial t}=-\frac{\partial\boldsymbol{z}}{\partial s}-iH(s)\boldsymbol{z},\ \boldsymbol{z}(0,s)=\delta(s)\tilde{\boldsymbol{w}}(0),
\end{equation*}
where $\delta(s)$ is the Dirac function.
By truncating the $s$-region to $[L',R']$, we can define the DFT matrices corresponding to the $s$ variable similar to that of $p$.
After the discrete Fourier spectral discretization on $s$, it yields
\begin{equation*}
i\frac{d\tilde{\boldsymbol{z}}}{dt}=(D_s\otimes I+I_{M_s}\otimes H)\tilde{\boldsymbol{z}},\ \tilde{\boldsymbol{z}}(0)=(\Phi_s^{-1}\otimes I)(\boldsymbol{\delta}_s\otimes\tilde{\boldsymbol{w}}(0)),
\end{equation*}
where $I=I_{2n}\otimes I_{M_p}$, $\boldsymbol{\delta}_s=\sum_j\zeta_{\omega}(s_j)\ket{j}$ and $\zeta_{\omega}$ is the regularizing function of $\delta$.
For $x\in\mathbb{R}$, it usually takes
$\zeta_{\omega}(x)=\frac{1}{\omega}\zeta(\frac{x}{\omega})$ with $\zeta(x)=1-\vert x\vert$ or $\zeta(x)=1/2(1+\cos(\pi x))$ whose support is $[-1,1]$.

\subsection{Complexity of the Schr{\"o}dingerisation}

At the end of this section, we give the complexity analysis of the Schr\"odingerisation~\cite{Jin2025queries}.
Suppose $R=-L$ is large enough satisfying
$$e^{L+\lambda_{\max}^-(H_1)T}\approx e^{L+\lambda_{\max}^+(H_1)T}\approx 0,$$
and $\Delta p = \mathcal{O}(\mu_{\max})$ is small, it follows that
\begin{equation}\label{eq:spectral_err}
\|\bm{u}_h(T)- \bm{u}(T)\|\leq \frac{\epsilon}{2} \|\bm{u}(T)\|,
\end{equation}
where $\bm{u}_h$ is the numerical solution after recovery in Section~\ref{retrieval}.
	
\begin{theorem}\label{thm:complexity}
Assume $R=-L$ are large enough and
$\Delta p=\mathcal{O}(\mu_{\max})$ is small enough to satisfy \eqref{eq:spectral_err}.
There exits a quantum algorithm that prepares an $\epsilon$-approximation of the state $\ket{\bm{u}(T)}$ with
$\Omega(1)$ success probability and a flag indicating success, using
\begin{equation*}
\tilde{\mathcal{O}}\bigg( \frac{\|\bm{u}(0)\|+T\|\bm{b}\|_{\text{smax}}}{\|\bm{u}(T)\|}\big( \alpha_H T \mu_{\max}+ \log\frac{\mu_{\max}(\|\bm{u}(0)\|+T\|\bm{b}\|_{\text{smax}})}{\epsilon\|\bm{u}(T)\|}\big)\bigg)
\end{equation*}
queries to the oracle for Hamiltonian simulation, where $\alpha_H \geq \|H_i\|, i=1,2$, and using
\begin{equation*}
\mathcal{O}\Big(\frac{\|\bm{u}(0)\|+T\|\bm{b}\|_{\text{smax}}} {\|\bm{u}(T)\|}\Big)
\end{equation*}
queries to the state preparation oracle for $\tilde{\bm{w}}_0$, where $\|\bm{b}\|_{\text{smax}}$ is defined by
$$\|\bm{b}\|_{\text{smax}}^2 = \sum_i \Big(\sup_{t\in [0,T]} |b_i(t)| \Big)^2.$$
\end{theorem}

Moreover, the smooth extension to $p<0$ for $e^{-p}$ is sufficient to achieve optimal dependence on matrix queries.
\begin{theorem}\label{thm:complexity2}
Suppose that the inhomogeneous dynamical system in \eqref{ODEs} is autonomous.  
Under the conditions of Theorem \ref{thm:complexity}, assume that $\xi(p)$ in \eqref{linearconv} is sufficiently smooth. 
Then there is a quantum algorithm that prepares an $\epsilon$-approximation of the normalized solution $\ket{\bm{u}(T)}$ with $\Omega(1)$ probability and a flag indicating success, using
\[\widetilde{\mathcal{O}}\Big(\frac{\|\bm{u}(0)\|+T\|\bm{b}\|_{\text{smax}}}{\|\bm{u}(T)\|} \alpha_H  T \log(1/\epsilon)\Big)\]
queries to the oracle for Hamiltonian simulation and
\[\mathcal{O}\Big(\frac{\|\bm{u}(0)\|+T\|\bm{b}\|_{\text{smax}}}{\|\bm{u}(T)\|}  \Big) \]
queries to the state preparation oracle for $[\xi(p_0), \cdots, \xi(p_{M_p-1})]^\top \otimes \bm{u}_0$.
\end{theorem}
For classical computation of $d$-dimensional differential equations, the complexity typically exhibits exponential dependence on $d$. 
However, as demonstrated in the above theorem, the $d$-dependence in Schr{\"o}dingerisation's complexity stems solely from sparsity $s$ (included in $\alpha_H$) and logarithmic components. 
This dependence is at most polynomial in $d$. Consequently, quantum algorithms demonstrate significant advantages in time complexity for high-dimensional problems, not to mention their inherent superiority in space complexity.

Thus, implementing quantum versions of dimension-increasing classical algorithms achieves dual benefits: complexity improvement and preservation of the native scheme's superior properties.

\section{One-dimensional scalar transport equations}\label{one-dim}

We first review the  geometrical optics (GO) based method introduced in~\cite{Nicolas2017Nonlinear} for the deterministic one dimensional scalar equation,
\begin{equation}\label{eq1}
\partial_t u+c(x)\partial_x u+\lambda u=\frac{ia(x)}{\varepsilon}u,\quad u(0,x)=u_0(x),
\end{equation}
where $u(t,x)\in\mathbb{C},x\in\Omega_x,t\ge 0$, and $\lambda$ is a complex constant.
The functions $u_0,a$ and $c$ are given while periodic boundary conditions are considered in space.
Furthermore, function $a$ satisfies $a(x)\ge a_0>0,\forall x\in\Omega_x$.
We will also allow the initial data to be
oscillatory
$$u_0(x)=f_{0}(x,\beta(x)/\varepsilon)\equiv f_{0}(x,\tau)\quad \textit{with}\quad\tau=\frac{\beta(x)}{\varepsilon},$$
where $\beta$ is given and $f_{0}$ is assumed to be periodic w.r.t. variable $\tau$.
Then we expand the initial data with regard to the periodic variable $\tau$:
\begin{equation}\label{u0expand}
u_0(x)=\sum_{k\in\mathbb{Z}}f_k(x)e^{ik\beta(x)/\varepsilon}\equiv \sum_{k\in\mathbb{Z}} u_k(0,x),
\end{equation}
and after applying to \eqref{eq1} one can get the following equations
\begin{equation}\label{eq2}
\partial_t u_k+c(x)\partial_x u_k+\lambda u_k=\frac{ia(x)}{\varepsilon}u_k,\quad u_k(0,x)=f_k(x)e^{ik\beta(x)/\varepsilon}.
\end{equation}
Indeed the linearity of equation \eqref{eq1} allows the use of the superposition principle which means the solution can be recovered by $u(t,x)=\sum_k u_k(t,x)$.
Then we apply the standard Geometric Optics (GO) method by
injecting the ansatz
\begin{equation}\label{GOansatz}
u_k(t,x)=\alpha_k(t,x)e^{iS_k(t,x)/\varepsilon}
\end{equation}
into \eqref{eq2} to get
\begin{equation*}
\partial_t \alpha_k+c(x)\partial_x \alpha_k+\lambda\alpha_k+\frac{i}{\varepsilon}\left[\partial_t S_k+c(x)\partial_x S_k\right]\alpha_k=\frac{ia(x)}{\varepsilon}\alpha_k.
\end{equation*}
To remove the terms in $1/\varepsilon$, one can impose the following equations on $\alpha_k$ and $S_k$
\begin{subequations}
\begin{align}
\partial_t \alpha_k+c(x)\partial_x \alpha_k+\lambda\alpha_k=0,\quad \alpha_k(0,x)=f_k(x),\label{aimeqs1}\\
\partial_t S_k+c(x)\partial_x S_k=a(x),\quad S_k(0,x)=k\beta(x).\label{aimeqs2}
\end{align}
\end{subequations}
In practical computation, the series \eqref{u0expand} must be truncated. Since we are performing the Fourier expansion with respect to the fast variable $\tau$, the resulting truncation error is independent of $\varepsilon$ (i.e.  the number of terms is independent of $\varepsilon$).

\subsection{Quantum simulation of (\ref{aimeqs1}-\ref{aimeqs2})  with constant convection term}\label{QSCC}

Next we describes the quantum simulation of equation \eqref{eq1}.
First of all, we consider the case of equations (\ref{aimeqs1}-\ref{aimeqs2}) with $c(x)\equiv c$ and $c\in\mathbb{R}\setminus\{0\}$.
We choose an uniform mesh size for $x$ set by $M_x=2N_x$.
For convenience, we also apply the notations in Section~\ref{sec_DFT} to the variable $x$.
One just needs to modify the subscript $p$ to $x$, for
example, $\Phi_x$ represents the Fourier transformation matrix for $x$.
Considering the Fourier spectral discretisation on $x$ in \eqref{aimeqs1}, one easily has
\begin{equation}\label{const1}
\frac{d}{dt} {\boldsymbol{\alpha}_k}+icP_x{\boldsymbol{\alpha}_k}+\lambda \boldsymbol{\alpha}_k=\boldsymbol{0},
\end{equation}
where $\boldsymbol{\alpha}_k(t)=\sum_{j}\alpha_k(t,x_j)\ket{j}$.
Moreover, by denoting $\boldsymbol{c}_k=(\Phi_x)^{-1}\boldsymbol{\alpha}_k(t)$, one then gets
\begin{equation*}
\frac{d}{dt} {\boldsymbol{c}_k}+icD_x{\boldsymbol{c}_k}+\lambda \boldsymbol{c}_k=\boldsymbol{0}.
\end{equation*}
It can also be rewritten as
\begin{equation*}
\frac{d}{dt}\boldsymbol{c}_k=-i\left(cD_x+{\text{Im}}({\lambda})I_{M_x})\right)\boldsymbol{c}_k-{\text{Re}}({\lambda}) I_{M_x}\boldsymbol{c}_k.
\end{equation*}
Since the identity matrix is commutative with any matrix, then by using $\tilde{\boldsymbol{c}}_k(t)=e^{{\text{Re}}({\lambda})t}{\boldsymbol{c}_k}(t)$, we can get the following Hamiltonian system:
\begin{equation*}
i\frac{d}{dt}\tilde{\boldsymbol{c}}_k=\left(cD_x+{\text{Im}}({\lambda})I_{M_x})\right)\tilde{\boldsymbol{c}}_k,\ \tilde{\boldsymbol{c}}_k(0)=(\Phi_x)^{-1}\boldsymbol{\alpha}_k(0).
\end{equation*}

On the other hand, by considering the Fourier spectral discretisation on $x$ in \eqref{aimeqs2}, we can get
\begin{equation}\label{const2}
\frac{d}{dt}\boldsymbol{S}_k+icP_x\boldsymbol{S}_k=\boldsymbol{a},
\end{equation}
where $\boldsymbol{S}_k(t)=\sum_{j}S_k(t,x_j)\ket{j}$ and $\boldsymbol{a}=\sum_{j}a(x_j)\ket{j}.$
In the same way, let $\boldsymbol{d}_k=(\Phi_x)^{-1}\boldsymbol{S}_k(t)$ and $\tilde{\boldsymbol{a}}=(\Phi_x)^{-1}\boldsymbol{a}$. 
We then have
\begin{equation*}
\frac{d}{dt}\boldsymbol{d}_k+icD_x\boldsymbol{d}_k=\tilde{\boldsymbol{a}}.
\end{equation*}
For this case, we adopt the following approach.
By using $\tilde{\boldsymbol{d}}_k=cD_x\boldsymbol{d}_k+i\tilde{\boldsymbol{a}}$, we finally get the Hamiltonian system:
\begin{equation*}
i\frac{d}{dt} \tilde{\boldsymbol{d}}_k=cD_x\tilde{\boldsymbol{d}}_k,\ \tilde{\boldsymbol{d}}_k(0)=cD_x(\Phi_x)^{-1}\boldsymbol{S}_k(0)+i(\Phi_x)^{-1}\boldsymbol{a}.
\end{equation*}
If $D_x(j,j)=0$ for (at most) one $j$, it means $\bra{j}\frac{d}{dt} \boldsymbol{d}_k=\bra{j}\tilde{\boldsymbol{a}}$ such that $\bra{j}\boldsymbol{d}_k$ linearly evolves.
So we can just remove it to simulate a $(M_x-1)$-order equation or replace $D_x(j,j)$ with any non-zero real number.
In fact, \eqref{const1} and \eqref{const2} correspond to two types of ODEs
\begin{equation*}
\frac{d}{dt} x=iHx+\lambda I_n x,\ \frac{d}{dt} y=iHy+F.
\end{equation*}
where $\lambda\in\mathbb{C},x,y,F\in\mathbb{C}^n$ and $H\in\mathbb{C}^{n\times n}$ is a Hermitian matrix.
They can be transformed into Hamiltonian systems by using transformation $\tilde{x}=e^{-{\text{Re}}({\lambda})t}x$ (which also can be derived by Corollary~\ref{lambdachoice}) and $\tilde{y}=Hy-iF$ without adding any independent variable.
Finally, we denote $m_x$ to be the number of qubits on $x$ register and can obtain the following estimation.
\begin{theorem}\label{complex1}
For a constant convection term, the solution to the equations \eqref{aimeqs1} and \eqref{aimeqs2} can be simulated with gate complexity 
$$N_{Gates}=\mathcal{O}(m_{x}\log{(m_{x})}).$$
\end{theorem}

\begin{proof}
It is known that the quantum Fourier transforms on $x$ in one dimension can be implemented using $\mathcal{O}(m_x\log{m_x})$ gates.
Moreover, the diagonal unitary operators can be implemented using $\mathcal{O}(m_{x})$ gates.
Thus, the proof is completed through Lemma~\ref{LemSch}. 
\end{proof}

\subsection{Quantum simulation of (\ref{aimeqs1}-\ref{aimeqs2}) with general nonnegative convection term}\label{QSnl}

For non-constant convection term $c(x)\ge 0$, we cannot get a diagonal unitary operator by the way in section \ref{QSCC}.
We will use the Schr{\"o}dingerisation approach with a uniform mesh size for $p$ set by $M_p=2N_p$.

\subsubsection{Quantum simulation with the spectral method}

By applying the Fourier collocation spectral discretisation on $x$ in \eqref{aimeqs1} and \eqref{aimeqs2}, we can obtain
\begin{subequations}
\begin{align}
&\frac{d}{dt} {\boldsymbol{\alpha}_k}+iC(x)P_x{\boldsymbol{\alpha}_k}+\lambda \boldsymbol{\alpha}_k=\boldsymbol{0},\label{speeq1}\\
&\frac{d}{dt} {\boldsymbol{S}_k}+iC(x)P_x{\boldsymbol{S}_k}=\boldsymbol{a},\label{speeq2}
\end{align}
\end{subequations}
where $C(x)={\text{diag}} \left[c(x_0),c(x_1),\cdots,c(x_{M_x-1})\right]$.
Then, applying  the Schr{\"o}dingerisation method in section \ref{review} we can get the following Hamiltonian system from \eqref{speeq1}:
\begin{equation}\label{speeq1H}
i\frac{d}{dt}\boldsymbol{c}_k=(H_1\otimes D_p-H_2\otimes I_{M_p})\boldsymbol{c}_k:=H_{\alpha}\boldsymbol{c}_k,
\end{equation}
where $\boldsymbol{c}_k(0)=(I_{M_x}\otimes(\Phi_p)^{-1})\sum_{j,k} \xi(p_k)\alpha_k(0,x_j)\ket{j}\ket{k}$ and
\begin{align*}
&H_1=-{\text{Re}}(\lambda) I_{M_x}-i\frac{C(x)P_x-P_x C(x)}{2},\\
&H_2=-{\text{Im}}(\lambda) I_{M_x}-\frac{C(x)P_x+P_x C(x)}{2}.
\end{align*}

On the one hand, one can also obtain the Hamiltonian system from \eqref{speeq2}:
\begin{equation}\label{speeq2H}
i\frac{d}{dt}\boldsymbol{d}_k=(H_1\otimes D_p-H_2\otimes I_{M_p})\boldsymbol{d}_k:=H_{S}\boldsymbol{d}_k,
\end{equation}
where $\boldsymbol{d}_k(0)=(I_{M_x}\otimes(\Phi_p)^{-1})\sum_{j,k} \xi(p_k)S_k(0,x_j)\ket{j}\ket{k}$ and
$$
H_1=\frac{1}{2}\left[\begin{array}{cc}
    i(P_x C(x)-C(x)P_x)& A \\
    A  & \boldsymbol{0}
\end{array}\right],$$
$$ 
H_2=\frac{1}{2i}\left[\begin{array}{cc}
    -i(C(x)P_x+P_x C(x)) &  A\\
     -A & \boldsymbol{0} 
\end{array}\right],\ A={\text{diag}}\{\boldsymbol{a}\}.
$$

It should be mentioned that principal diagonal of $i(P_x C(x)-C(x)P_x)$ are all zero.
Thus, for estimations of the maximal eigenvalue, one can use Gershgorin’s Theorem or other effective ways.
Then the technique in Corollary~\ref{lambdachoice} can be used by choosing an appropriate $\lambda_0$.
After the Hamiltonian simulations of \eqref{speeq1H} and \eqref{speeq2H}, $\boldsymbol{\alpha}_k(T)$ and $\boldsymbol{S}_k(T)$ can be recovered according to \eqref{restore}.
In the rest of paper, we won't repeat this explanation again and will point out explicit eigenvalue if it is possible.
For a description of Schr{\"o}dingerisation method, we express the above process by the following Algorithm 1.
\begin{algorithm}[ht]
\caption{}
\hspace*{0.02in} {\bf Input:}
$u_0(x)$
\begin{algorithmic}[1]
\State Give the initial data $\alpha_k(0,x)$ and $S_k(0,x)$ according to $u_0(x)$.
\State Encode initial data into initial states $\ket{\tilde{\psi}_c(0)}=\frac{\boldsymbol{c}_k(0)}{\Vert\boldsymbol{c}_k(0)\Vert},\ket{\tilde{\psi}_d(0)}=\frac{\boldsymbol{d}_k(0)}{\Vert\tilde{\boldsymbol{d}}_k(0)\Vert}.$
\State Perform Hamiltonian simulation to implement the unitary operators $e^{-iH_{\alpha}T},e^{-iH_{S}T}$ on initial states to get $\ket{\tilde{\psi}_c(T)}$ and $\ket{\tilde{\psi}_d(T)}$.
\State Apply QFT on the resulting state to retrieve $\ket{\psi_c(T)}$ and $\ket{\psi_d(T)}$.
\State Conduct measurement in section \ref{retrieval} and use \eqref{restore} to recover the numerical solutions of $\alpha_k^d(T,x)$ and $S_k^d(T,x)$ at interested points $x$.
\State Restore the numerical solution $u^d(T,x)$ due to the ansatz \eqref{GOansatz} in classic computer.
\end{algorithmic}
\hspace*{0.02in} {\bf Output:}
$u^d(T,x)$ or other required data
\end{algorithm}

\begin{remark}
 In step 5, we take measurements on $\alpha_k^d(T,x)$ and $s_k^d(T,x)$ at interested points $x$ which should be a small subset of the entire computational domain. Then, in Step 6, we reconstruct  the final solution $u^d$  classically.To take measurements for all spatial grid points $x$ will involve exponential cost which is still the bottleneck for quantum algorithms for PDEs. 
\end{remark}

Assume that we use a sufficiently smooth $\xi(p)$.
By applying Theorem~\ref{thm:complexity2}, we can get that the complexity of the Schr{\"o}dingerisation for computing $\alpha_k$ needs
\[\widetilde{\mathcal{O}}\Big(\frac{T(\vert\lambda\vert+M_x\Vert c(x)\Vert_{\ell^{\infty}})\|\bm{\alpha}_k(0)\|}{\|\bm{\alpha}_k(T)\|}\log(1/\epsilon)\Big)\]
queries, where $\epsilon$ is the tolerance error.
For $S_k$, it uses
\[\widetilde{\mathcal{O}}\Big(\frac{\|\bm{S}_k(0)\|+T\|\bm{a}\|_{\text{smax}}}{\|\bm{S}_k(T)\|} (\Vert a(x)\Vert_{\ell^{\infty}}+M_x\Vert c(x)\Vert_{\ell^{\infty}})  T \log(1/\epsilon)\Big)\]
queries.
Since the result follows immediately from Theorem~\ref{thm:complexity2} by direct application, we omit the remaining details and will not reiterate similar discussions hereafter.

%
%

\subsubsection{Quantum simulation with the finite difference discretisation}\label{QSFDM}

On the other hand, one can also apply the finite difference approximation to (\ref{aimeqs1}-\ref{aimeqs2}). 
For example, if the upwind scheme is used, one gets  
\begin{align*}
&\frac{d\alpha_k^{j}(t)}{dt}+c(x_j)\frac{\alpha_k^{j}-\alpha_k^{j-1}}{\Delta x}+\lambda \alpha_k^{j}=0,\quad \alpha_k^{j}(0)=f_k(x_j),\\
&\frac{dS_k^{j}(t)}{dt}+c(x_j)\frac{S_k^{j}-S_k^{j-1}}{\Delta x}=a(x_j),\quad 
 S_k^{j}(0)=k\beta(x_j).
\end{align*}
An increment operation can be defined as follows:
$$S^{+}:=\sum_{j=1}^{M_x-1}\ket{j}\bra{j-1}=\sum_{j=1}^{m_x} I_{2^{m_x-j}}\otimes \sigma_{10}\otimes\sigma_{01}^{\otimes(j-1)},\ S^{-}=(S^{+})^{\dagger}$$
with $2\times 2$ matrices $\sigma_{ij}=\ket{i}\bra{j}$.
Then by denoting $\boldsymbol{\alpha}_k(t)=\sum_j\alpha_k^{j}(t) \ket{j}, \boldsymbol{S}_k=\sum_j S_k^{j} \ket{j}$ and 
$$
D_x^{-}=\frac{1}{\Delta x}\left(I_{M_x}-S^{+}-\sigma_{01}^{\otimes m_x}\right),D_x^{+}=-(D_x^{-})^\dagger,
$$ 
we can write the above equation in vector form as
\begin{subequations}
\begin{align}
&\frac{d\boldsymbol{\alpha}_k}{dt}=M_1\boldsymbol{\alpha}_k\label{FDsys1}\\
&\frac{d\boldsymbol{S}_k}{dt}=M_2\boldsymbol{S}_k+F_k,\label{FDsys2}
\end{align}
\end{subequations}
where $M_1=-C(x)D_x^{-}-\lambda I_{M_x}$, $M_2=-C(x)D_x^{-}$ and $F_k=\sum_j a(x_{j})\ket{j}$.
Furthermore, one has
\begin{align*}
&H_1=-{\text{Re}}(\lambda) I_{M_x}+\frac{D_x^{+}C(x)-C(x)D_x^{-}}{2},\\
&H_2=-{\text{Im}}(\lambda) I_{M_x}-\frac{D_x^{+}C(x)+C(x)D_x^{-}}{2i}
\end{align*}
for \eqref{FDsys2} while 
$$
H_1=\frac{1}{2}\left[\begin{array}{cc}
    D_x^{+}C(x)-C(x)D_x^{-} & A \\
    A  & \boldsymbol{0}
\end{array}\right],
$$
$$ 
H_2=\frac{1}{2i}\left[\begin{array}{cc}
    -D_x^{+}C(x)-C(x)D_x^{-} &  A\\
     -A & \boldsymbol{0} 
\end{array}\right],\ A={\text{diag}}\{F_k\}
$$
for \eqref{FDsys2}.
%

\subsection{Numerical results}

Next we will present some tests to ensure the effectiveness of our quantum algorithms.
In this paper, the Hamiltonian simulation is implemented by the implicit midpoint rule which is also a symplectic method of order $2$.
In addition, we use \eqref{xip} for initializations.
Furthermore, we denote $M_x=2^m$ and $M_p=2^n$ in this section for later use.

\subsubsection{A constant convection term case}\label{numres1}
First we consider a constant convection case for benchmark with $c(x)=\lambda=a(x)=1$ and the following initial data
\begin{equation}\label{eg1initial}
u_0(x)=1+\frac{1}{2}\cos(2x)+i(1+\frac{1}{2}\sin(2x)),\quad x\in[-\frac{\pi}{2},\frac{\pi}{2}].
\end{equation}
The periodic boundary condition is chosen such that the exact solution is given by
$u(t,x)=\exp(-(1-i/\varepsilon)t)u_0(x-t)$.
Figure~\ref{imgeg11} shows the comparison between the exact solution and the numerical solution for different values of $\varepsilon$ at the final time $T=1$.
It can be observed that the numerical solutions match the accurate solutions well on all three scales.

\begin{figure}[!ht]
\centering
\subfigure[]{
\includegraphics[scale=.33]{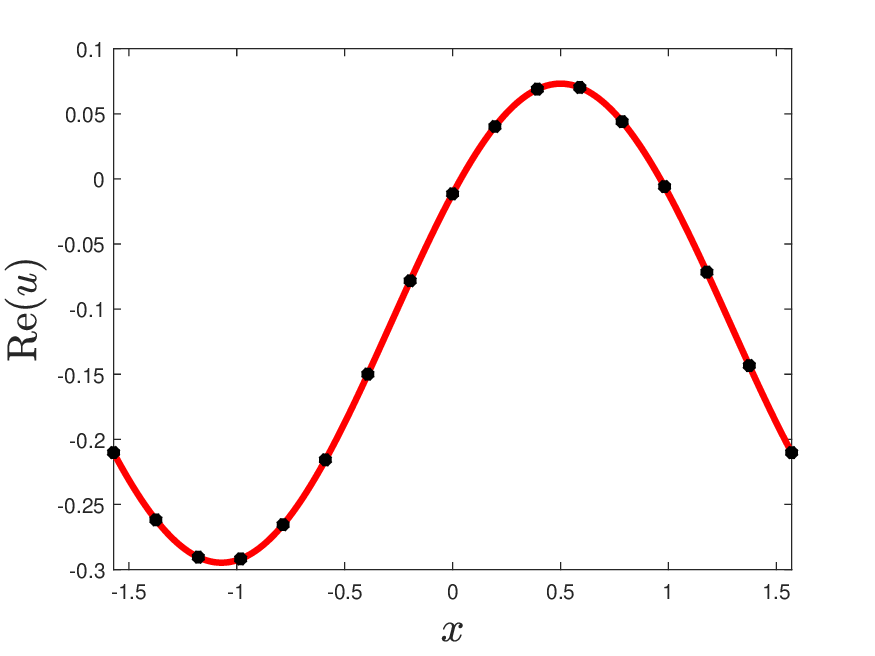}
}
\subfigure[]{
\includegraphics[scale=.33]{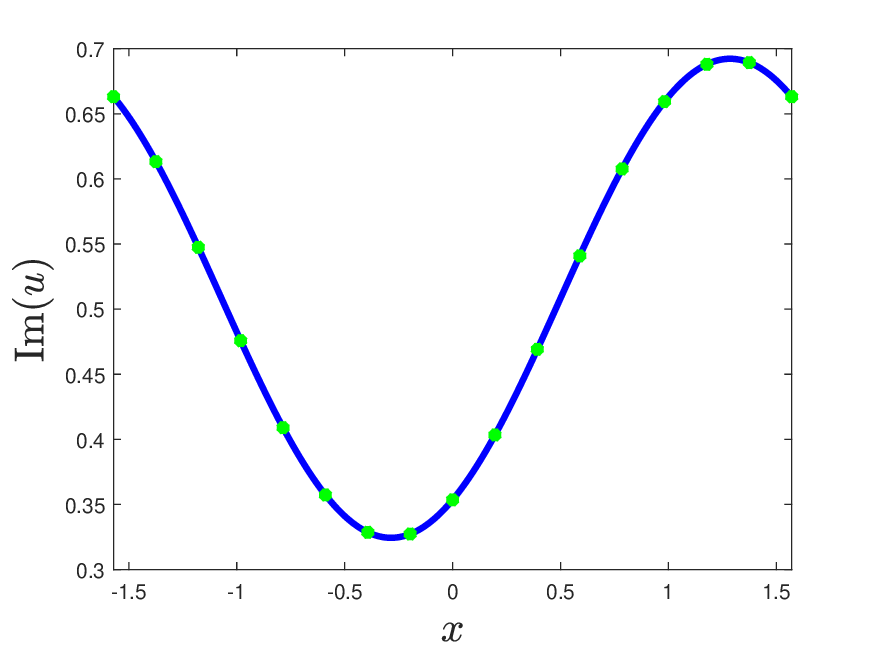}
}
\subfigure[]{
\includegraphics[scale=.33]{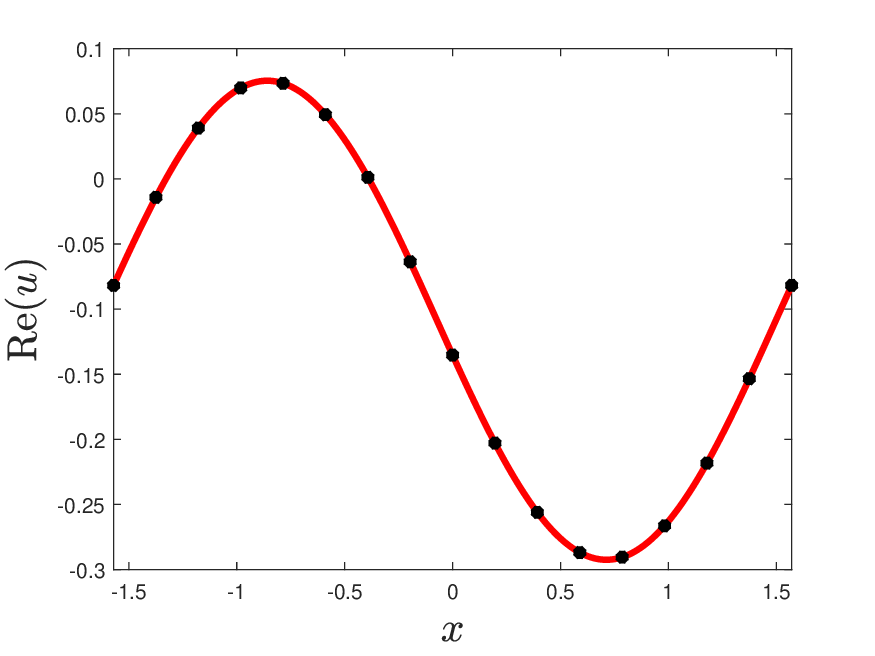}
}
\subfigure[]{
\includegraphics[scale=.33]{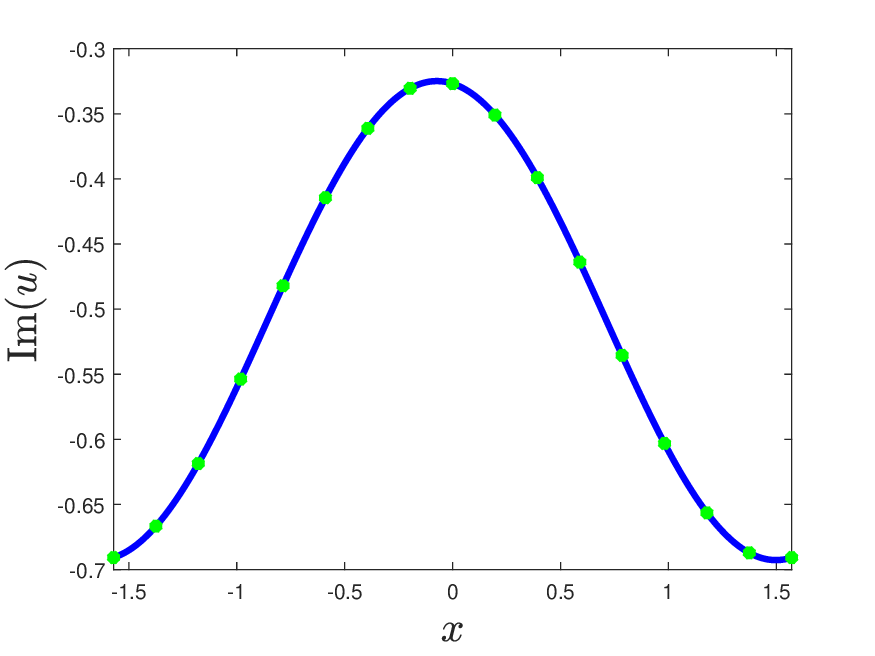}
}
\subfigure[]{
\includegraphics[scale=.33]{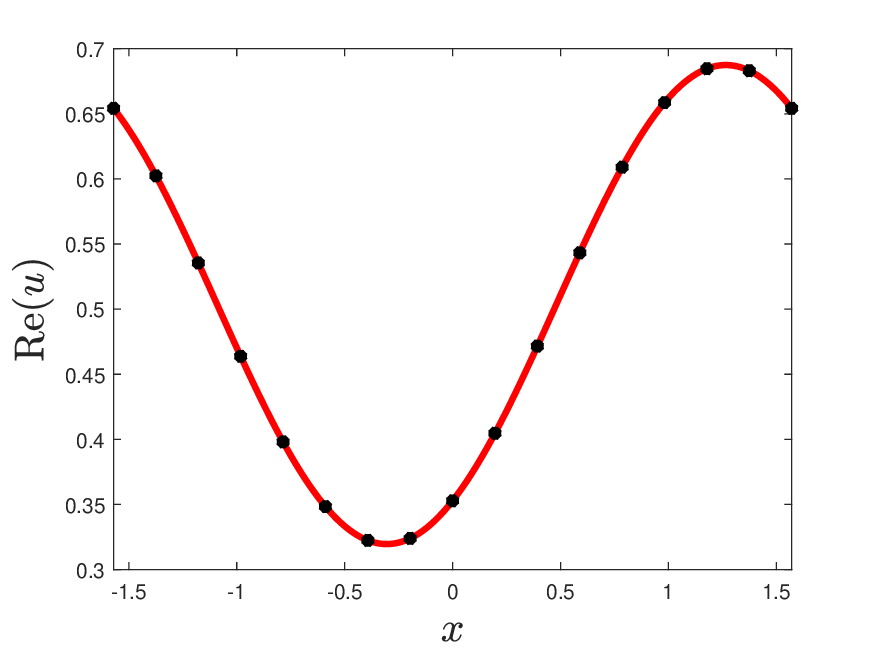}
}
\subfigure[]{
\includegraphics[scale=.33]{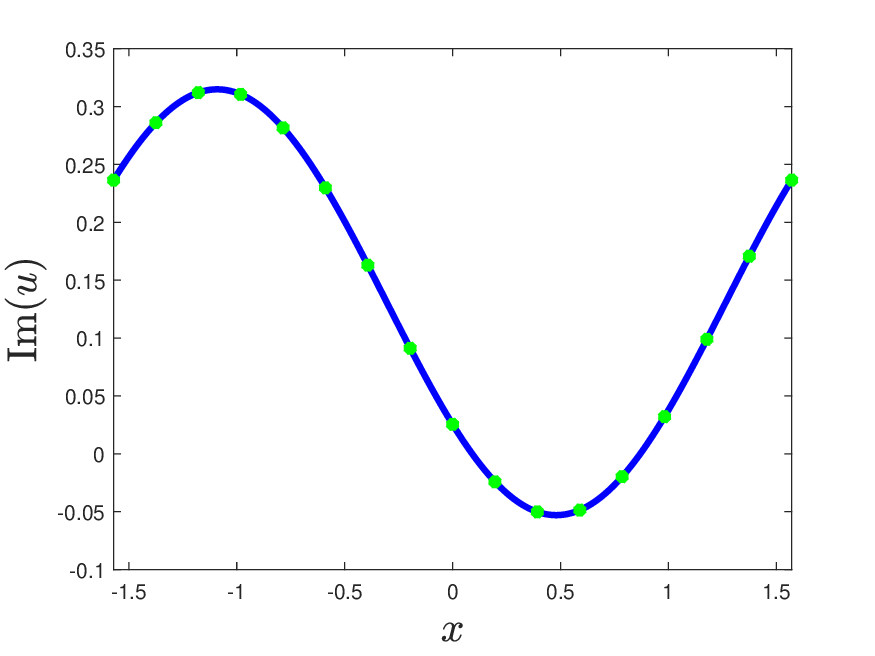}
}
\caption{Comparison between  the exact solution and the numerical solution for $T=1,m=4$. (a)(c)(e): Real part of $u$. (b)(d)(f): Imaginary part of $u$. (a)(b): $\varepsilon=1$. (c)(d): $\varepsilon=0.1$. (e)(f): $\varepsilon=0.01$.}
\label{imgeg11}
\end{figure}

To further display the capability of numerical method for capturing oscillations generated from both initial data and source--even without numerically resolving the oscillations, we add an oscillatory dependence with the initial phase $x/\varepsilon$ in Figure~\ref{imgeg12}.
The exact solution under periodic boundary condition and initial data
$$u_0(x)=e^{i\frac{x}{\varepsilon}}\left(1+\frac{1}{2}\cos(2x)+i(1+\frac{1}{2}\sin(2x))\right),\quad x\in[-\frac{\pi}{2},\frac{\pi}{2}].$$
is given by $\exp(-t+ix/\varepsilon)u_0(x-t).$ 
Since $\alpha_k$ and $S_k$ are not oscillatory, they can be solved quite efficiently by our algorithm. 

\begin{figure}[!ht]
\centering
\subfigure[]{
\includegraphics[scale=.3]{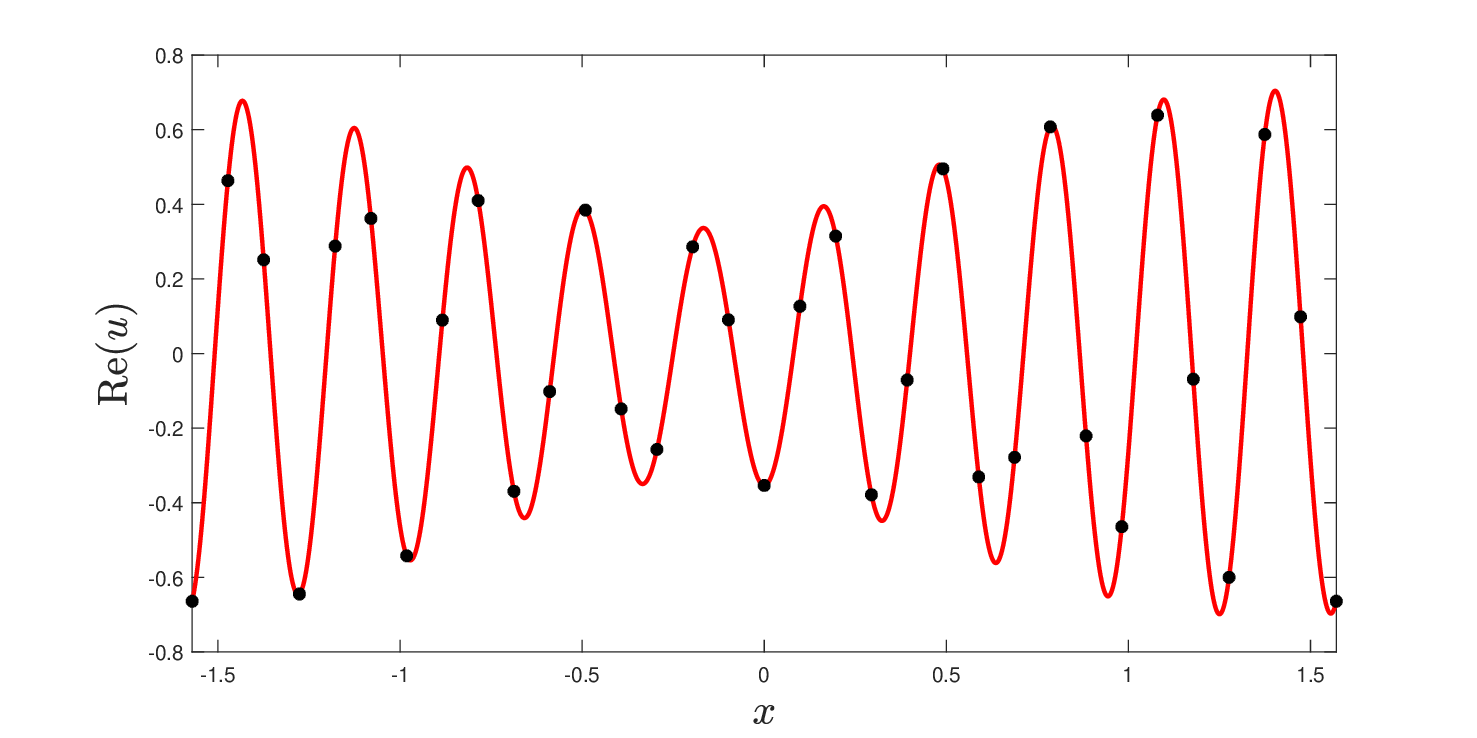}
}
\subfigure[]{
\includegraphics[scale=.3]{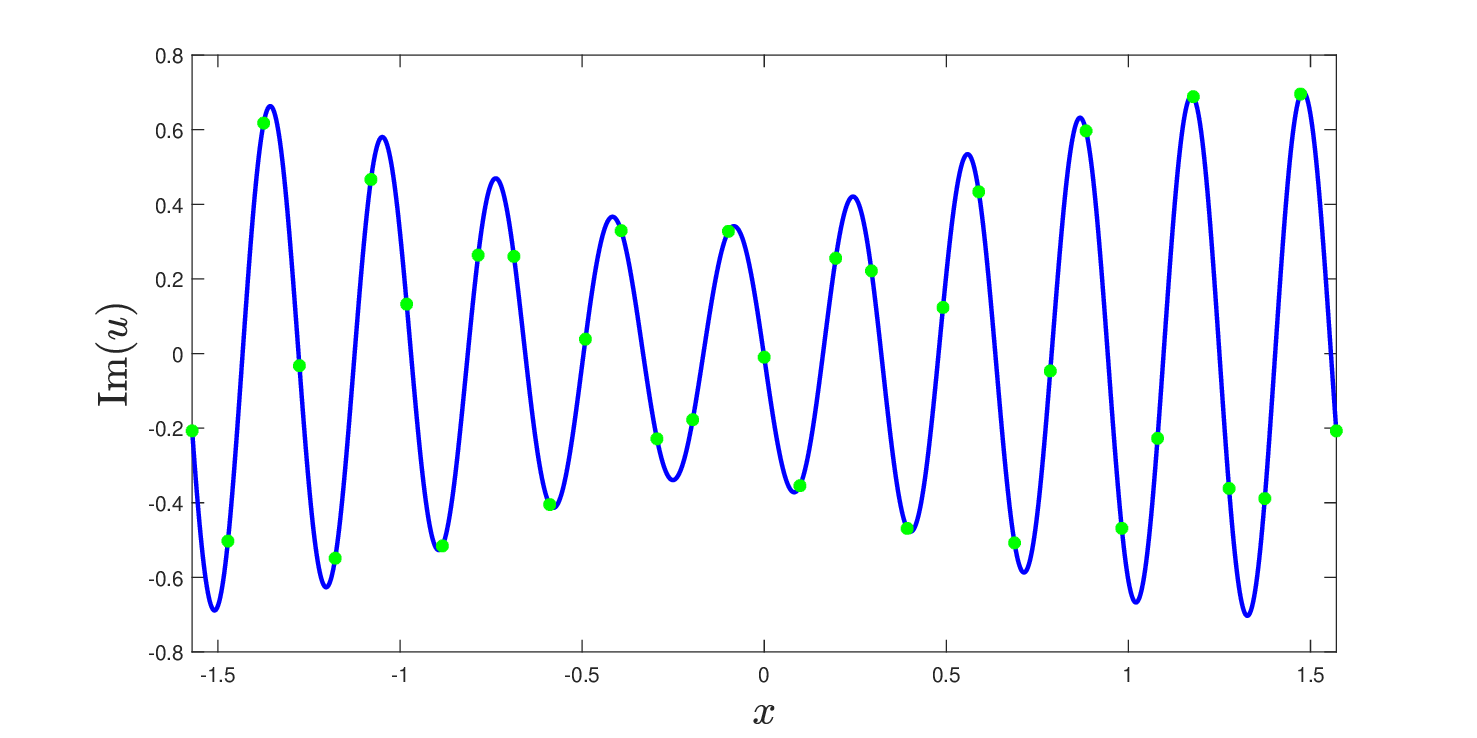}
}
\subfigure[]{
\includegraphics[scale=.3]{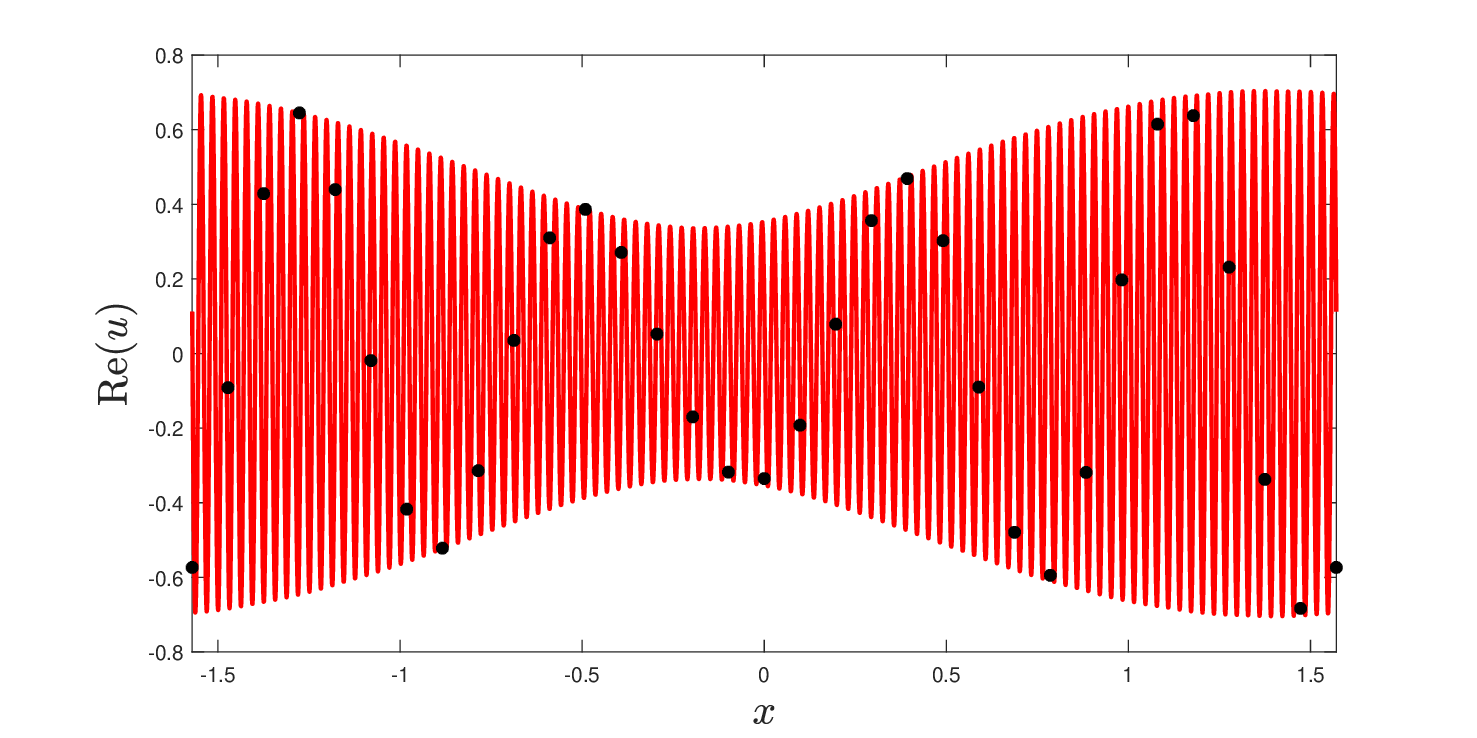}
}
\subfigure[]{
\includegraphics[scale=.3]{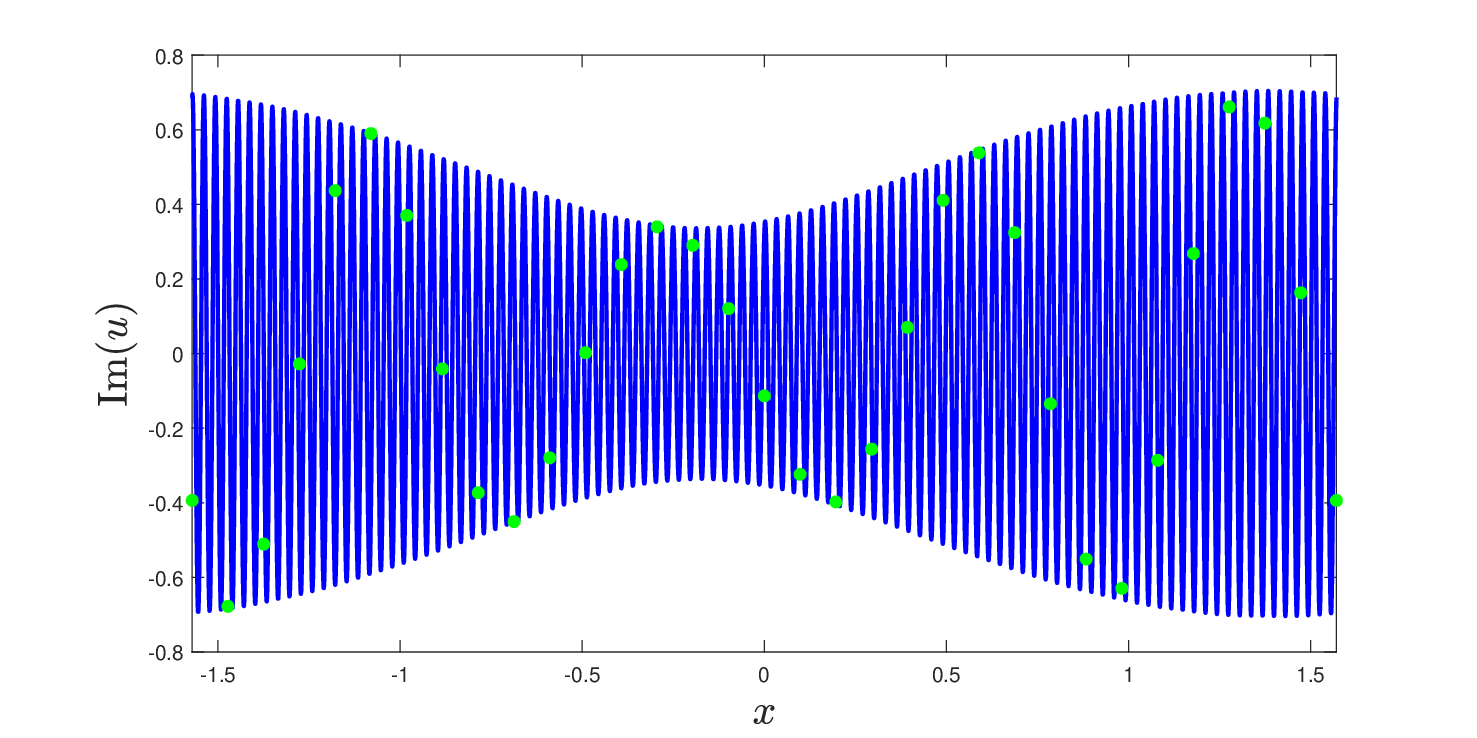}
}
\caption{Comparison between the exact solution and the numerical solution for $T=1,m=5$. (a)(c): Real part of $u$.
(b)(d): Imaginary part of $u$. (a)(b): $\varepsilon=0.1$. (c)(d): $\varepsilon=0.01$.}
\label{imgeg12}
\end{figure}

\subsubsection{A general nonnegative convection term case}

In this test, we solve \eqref{eq1} with $c(x)=\cos^2(x),a(x)=1.5+\cos(2x)$ with the initial data \eqref{eg1initial} at the final time $T=1$, and mesh size $m=5,n=9$.
For this case, the Schr{\"o}dingerisation method in section~\ref{QSnl} should be applied.
Since the spectral method and finite difference method have almost no different on performance in Figure~\ref{imgeg21}, we just plot one of them and illustrate the space-time oscillations arising in the solution with $a(x)$.
We can observe that our method is able to capture very well high oscillations, {\it point-wise},  in space--even without numerically resolving the oscillations, whether $S_k$ is accurately solved or not.

\begin{figure}[!ht]
\centering
\subfigure[]{
\includegraphics[scale=.3]{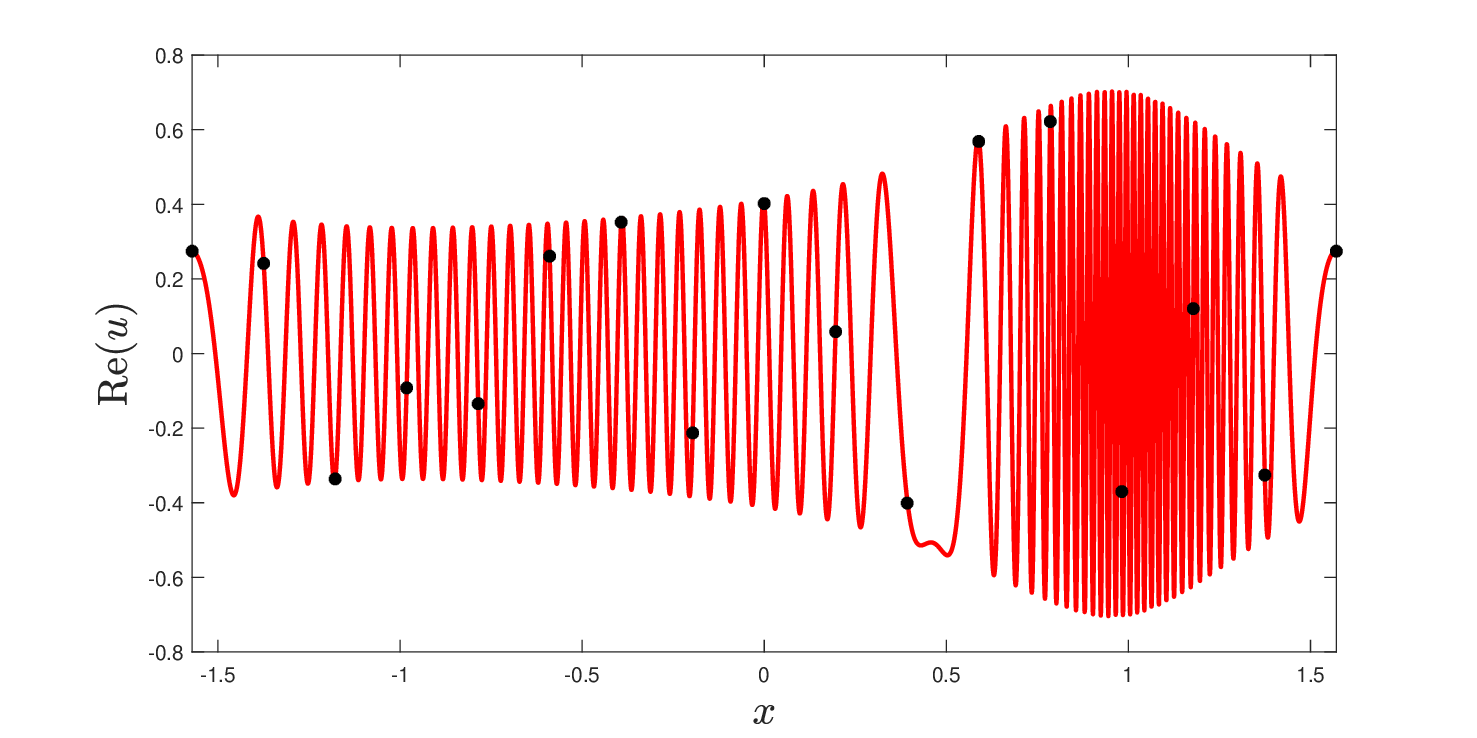}
}
\subfigure[]{
\includegraphics[scale=.3]{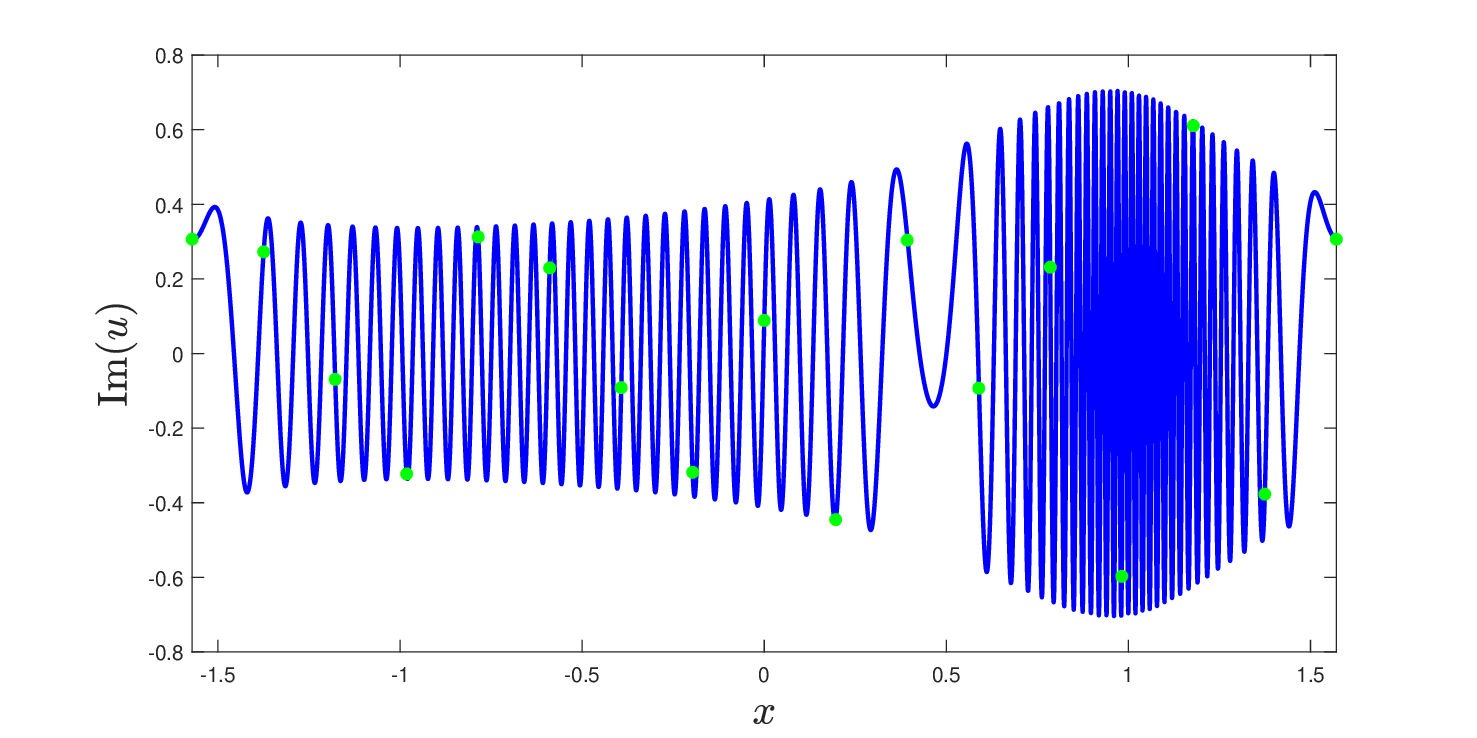}
}
\subfigure[]{
\includegraphics[scale=.3]{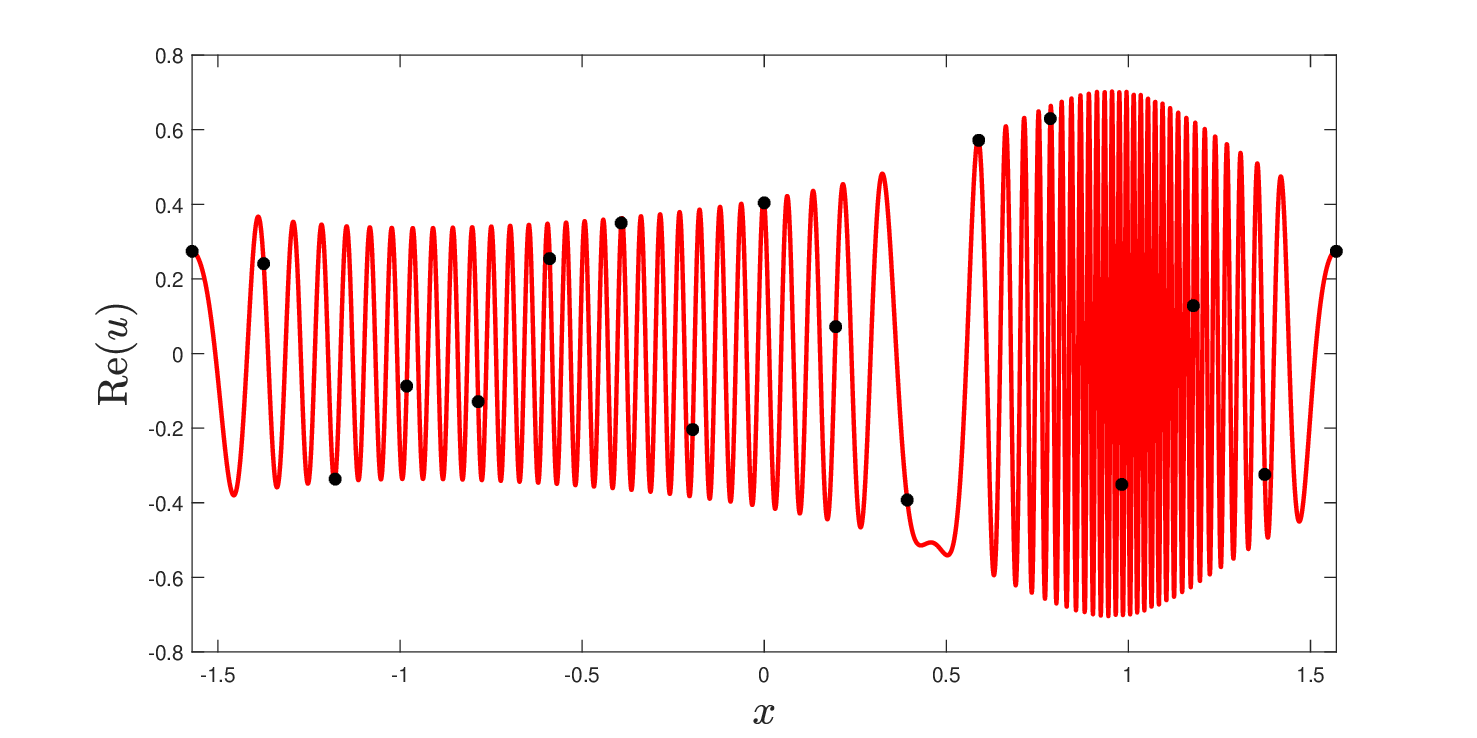}
}
\subfigure[]{
\includegraphics[scale=.3]{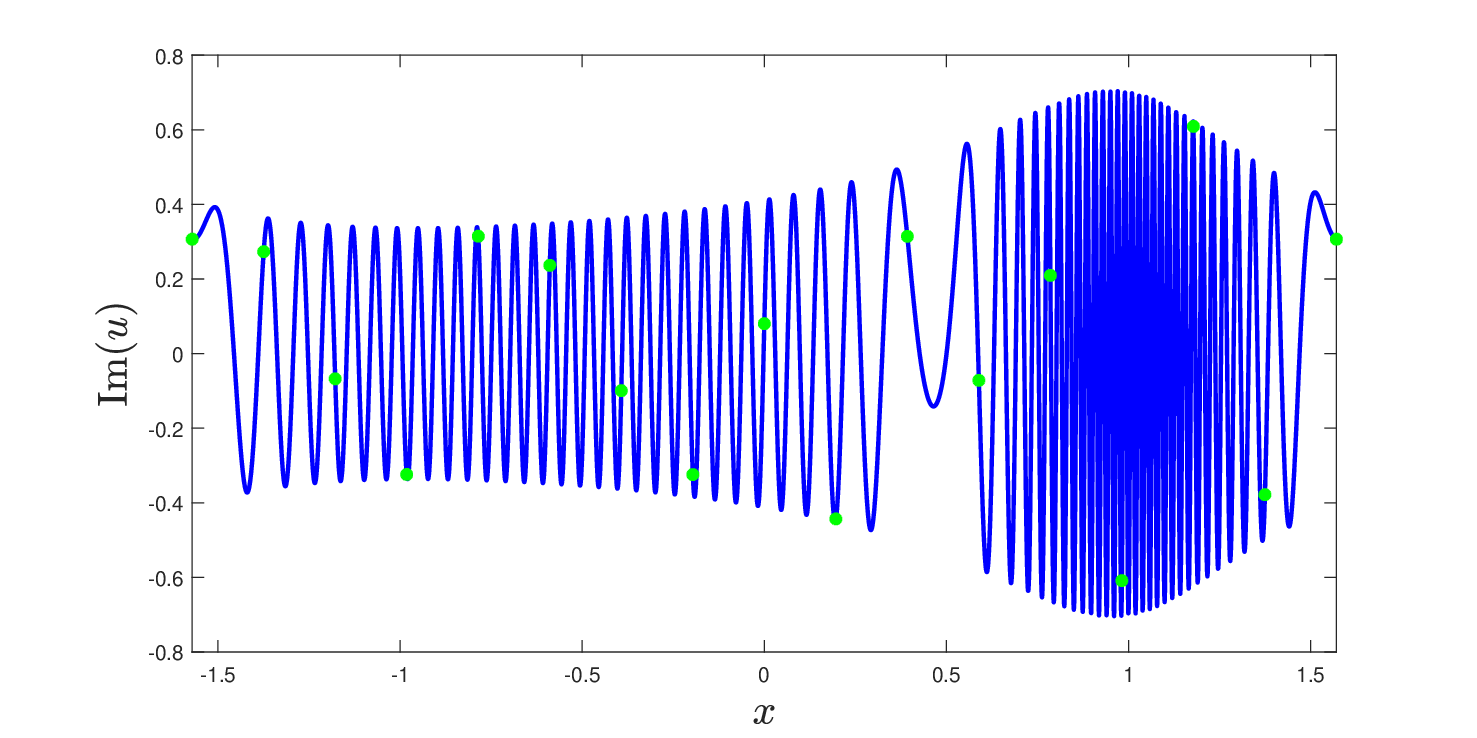}
}
\caption{Comparison between the exact solution and the solution of algorithm 1 for $\varepsilon=0.01,T=1,m=4,n=9,\lambda=1$. 
(a)(b): Numerical computation for $S_k$. (c)(d): Exact computation for $S_k$. (a)(c): Real part of $u$. (b)(d): Imaginary part of $u$.}
\label{imgeg21}
\end{figure}

In Figure~\ref{imgeg23}, we plot the error $\Vert\alpha_k^d(T,x)-\alpha_k(T,x)\Vert_{\infty}$ obtained by different recoveries.
The error is computed at the final time $T=1$ by mesh size $m=4,n=9$.
We choose $\lambda_0\in[-6-\lambda,6-\lambda]$ while the computational domain of $p$ is $[-10,10]$ due to $\lambda_n(H_1)=113/30-\lambda$ and $\lambda_1(H_1)=-113/30-\lambda$.
And in Figure~\ref{imgeg24}, we plot the bar graph of the number $\textit{card}\left(\{p_j: \Vert\alpha_k^d(T,x,p_j)-\alpha_k(T,x)\Vert_{\infty}\le\Delta p\}\right)$ for various $\lambda$ and $\lambda_0$.
It can be observed that the optimal recovery of Schr{\"o}dingerisation can be obtained by choosing a certain $\lambda_0$ which belongs to $[\frac{\lambda_1(H_1)+\lambda_n(H_1)}{2},\lambda_n(H_1)]$.


\begin{figure}[!ht]
\centering
\subfigure[]{
\includegraphics[scale=.5]{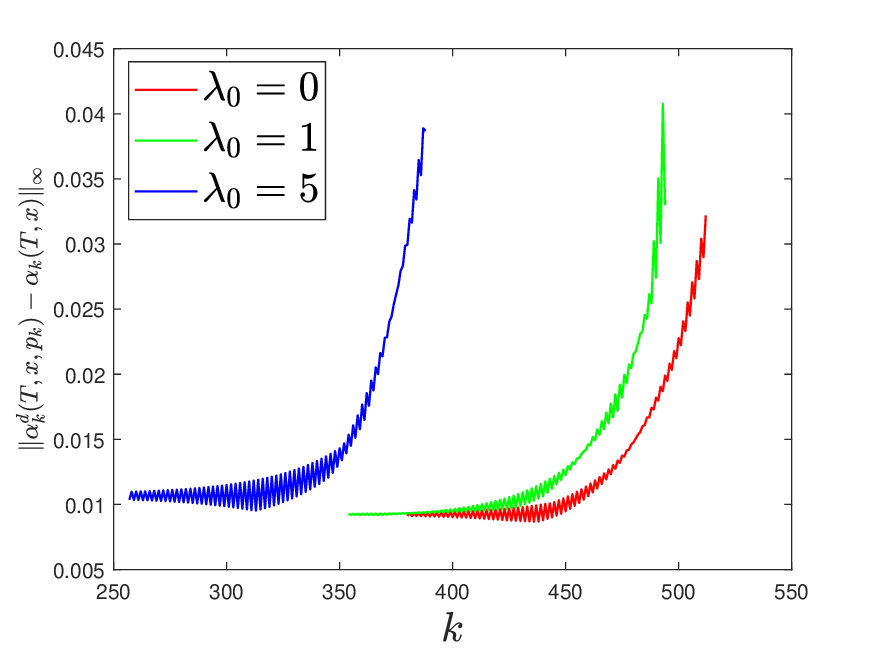}
}
\subfigure[]{
\includegraphics[scale=.5]{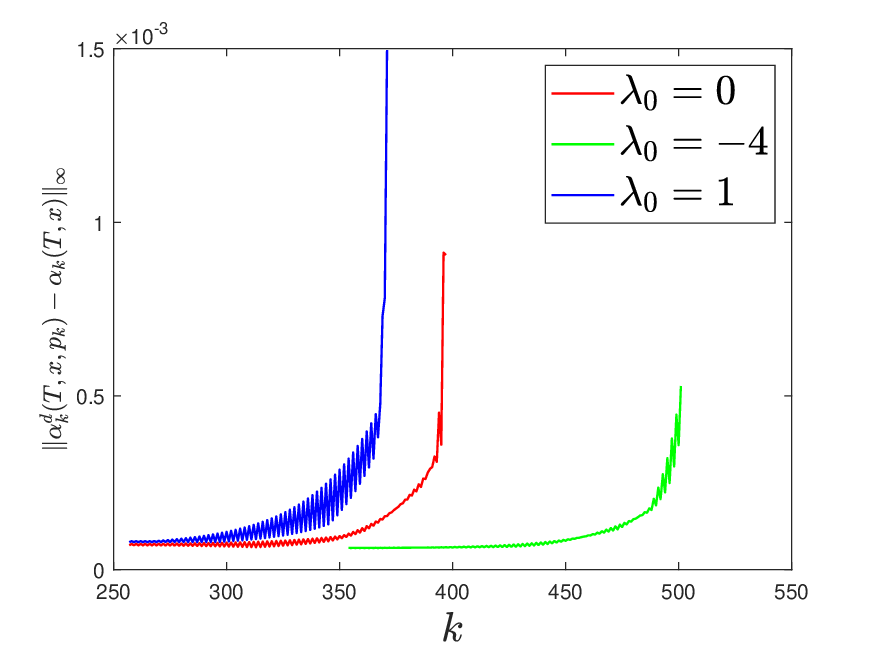}
}
\subfigure[]{
\includegraphics[scale=.5]{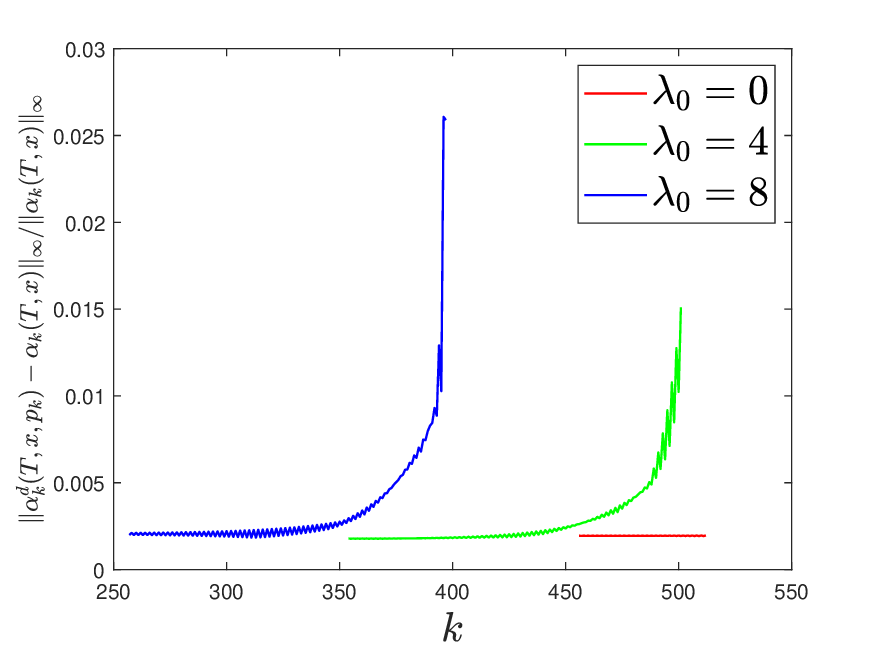}
}
\caption{Plot of the error estimation $\Vert\alpha_k^d(T,x,p_k)-\alpha_k(T,x)\Vert_{\infty}$ for $T=1,m=4,n=9$. (a): $\lambda=-1$ and $p_k\in\{p_j: \Vert\alpha_k^d(T,x,p_j)-\alpha_k(T,x)\Vert_{\infty}\le\Delta p\}$. (b): $\lambda=4$ and $p_k\in\{p_j: \Vert\alpha_k^d(T,x,p_j)-\alpha_k(T,x)\Vert_{\infty}\le\Delta p^2\}$. (c): $\lambda=-4$ and $p_k\in\{p_j: \Vert\alpha_k^d(T,x,p_j)-\alpha_k(T,x)\Vert_{\infty}\le\Delta p\Vert\alpha_k(T,x)\Vert_{\infty}\}$.}
\label{imgeg23}
\end{figure}

\begin{figure}[!ht]
\centering
\subfigure[]{
\includegraphics[scale=.3]{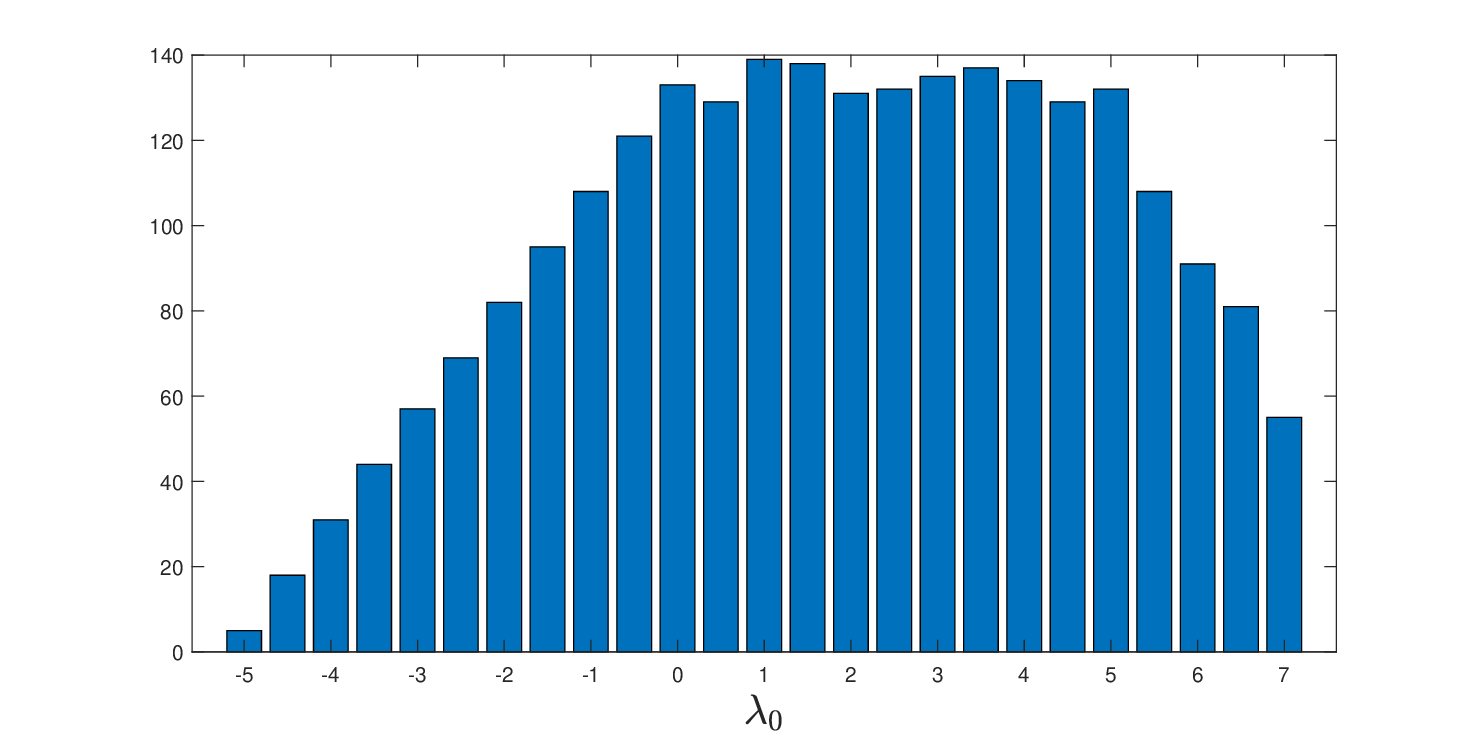}
}
\subfigure[]{
\includegraphics[scale=.3]{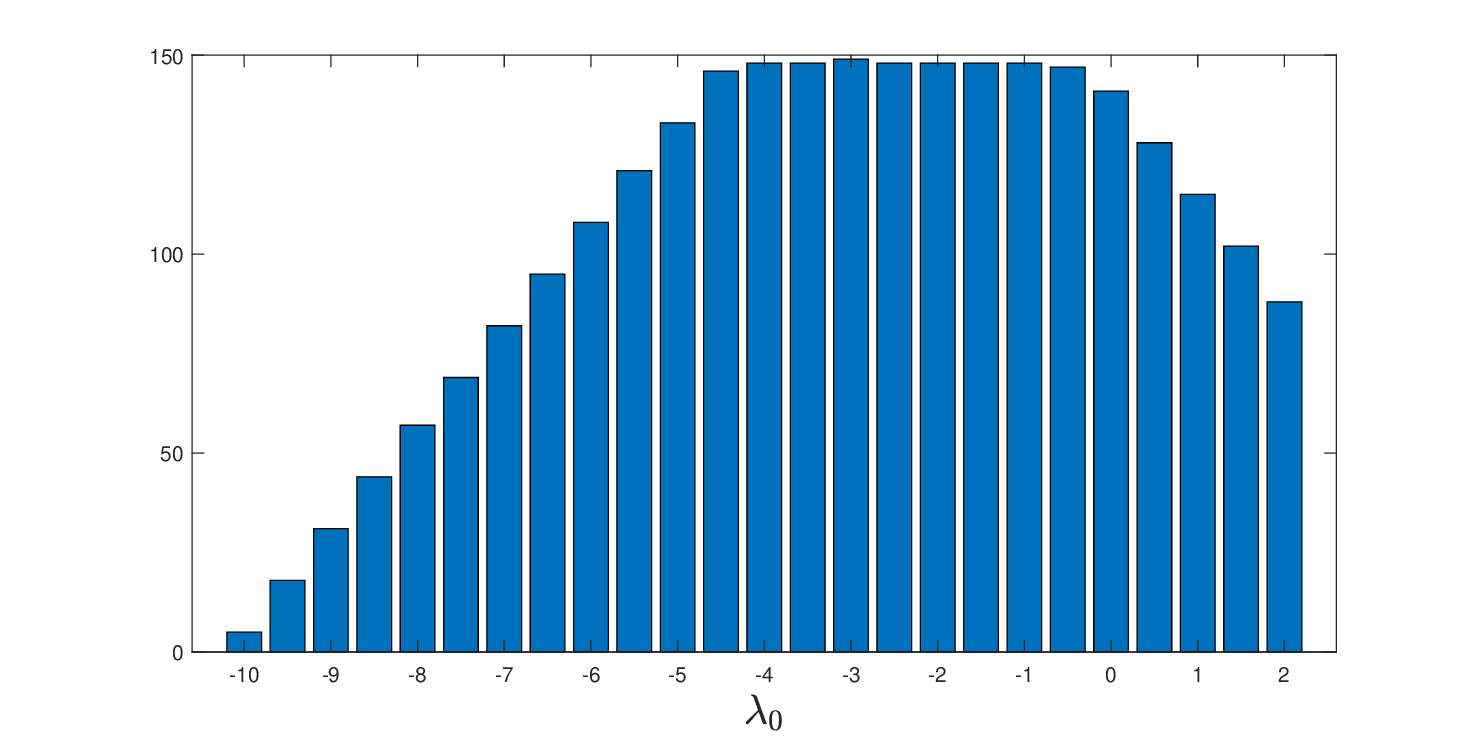}
}
\subfigure[]{
\includegraphics[scale=.3]{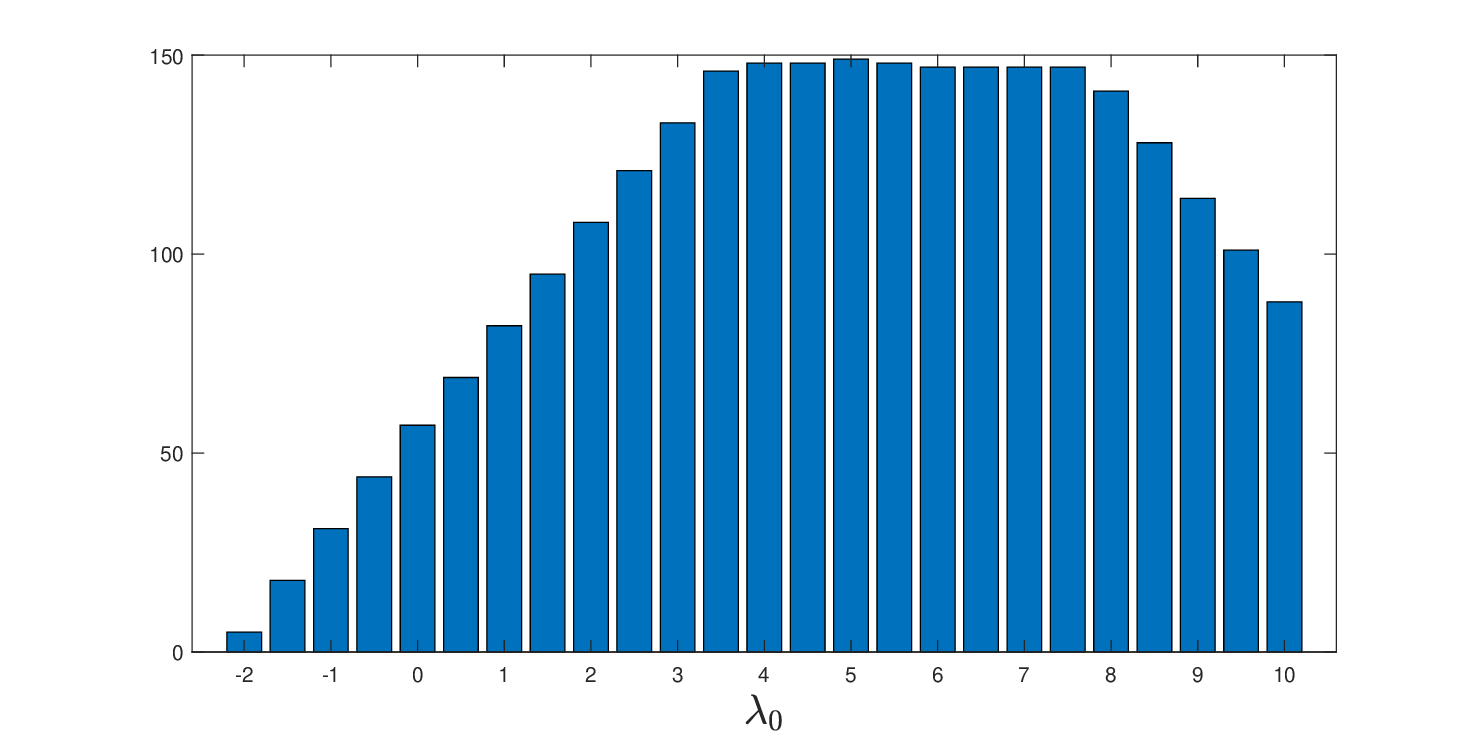}
}
\caption{Plot of the number $\textit{card}\left(\{p_j: \Vert\alpha_k^d(T,x,p_j)-\alpha_k(T,x)\Vert_{\infty}\le\textit{tol}\}\right)$ for $T=1,m=4,n=9$. (a): $\lambda=-1$, $\textit{tol}=\Delta p$. (b): $\lambda=4$, $\textit{tol}=\Delta p^2$. (c): $\lambda=-4$, $\textit{tol}=\Delta p\Vert\alpha_k(T,x)\Vert_{\infty}$.}
\label{imgeg24}
\end{figure}

\section{Vector-valued semi-classical transport equations}
In this section, we will first focus on the case where $u(t,x)\in\mathbb{C}^2,x\in[0,1],t\ge 0,$
\begin{equation}\label{exteqs}
\partial_t u+A(x)\partial_x u=\frac{iE(t,x)}{\varepsilon}Du+Cu,\quad u(0,x)=u_0(x).
\end{equation}
$E$ is a real scalar function, $C$ is a constant matrix and
\begin{equation*}
A(x)=\left[
            \begin{array}{cc}
              a_1(x) & 0 \\
              0 & a_2(x) 
            \end{array}
          \right],\quad
       D=\left[
            \begin{array}{cc}
              0 & 0 \\
              0 & -1
            \end{array}
          \right],\quad
          C=\left[
            \begin{array}{cc}
              C_{11} & C_{12} \\
              C_{21} & C_{22} 
            \end{array}
          \right].
\end{equation*}
To construct the uniformly accurate numerical scheme in~\cite{Nicolas2017Nonlinear}, we use the nonlinear geometric optics (NGO) ansatz:
$$U(t,x,S(t,x)/\varepsilon)=u(t,x).$$
Then, $U=(U_1,U_2)$ satisfies
\begin{align*}
&\partial_{t}U_1+a_1\partial_{x}U_1+\frac{1}{\varepsilon}\left[\partial_t S+a_1\partial_x S\right]\partial_{\tau} U_1=C_{11}U_1+C_{12}U_2,\\
&\partial_{t}U_2+a_2\partial_{x}U_2+\frac{1}{\varepsilon}\left[\partial_t S+a_2\partial_x S\right]\partial_{\tau} U_2=-\frac{iE}{\varepsilon}U_2+C_{21}U_1+C_{22}U_2.
\end{align*}
The equation for the phase $S$ writes
\begin{equation}\label{twodimSeq}
\partial_t S+a_2\partial_x S=E,\ S(0,x)=0,
\end{equation}
and the equations for $(U_1,U_2)$ become
\begin{align*}
&\partial_{t}U_1+a_1\partial_{x}U_1+\frac{1}{\varepsilon}\left[(a_1-a_2)\partial_x S+E\right]\partial_{\tau} U_1=C_{11}U_1+C_{12}U_2,\\
&\partial_{t}U_2+a_2\partial_{x}U_2=-\frac{E}{\varepsilon}\left[\partial_{\tau} U_2+iU_2\right]+C_{21}U_1+C_{22}U_2.
\end{align*}
Setting $V_2=e^{i\tau}U_2$, we finally obtain
\begin{equation}\label{exteqngo}
\begin{aligned}
&\partial_{t}U_1+a_1\partial_{x}U_1-C_{11}U_1-C_{12}e^{-i\tau}V_2=-\frac{1}{\varepsilon}\left[(a_1-a_2)\partial_x S+E\right]\partial_{\tau} U_1,\\
&\partial_{t}V_2+a_2\partial_{x}V_2-C_{21}e^{i\tau}U_1-C_{22}V_2=-\frac{E}{\varepsilon}\partial_{\tau} V_2.
\end{aligned}
\end{equation}

Also, \eqref{exteqngo} needs appropriate initial data $U_1(0,x,\tau)$ and $V_2(0,x,\tau)$.
A suitable initial condition with first order correction for them is provided using the Chapman-Enskog expansion~\cite{Nicolas2017Nonlinear}, and is given by 
\begin{equation}\label{NGOinitial}
\begin{aligned}
&U_1(0,x,\tau)=f_1^{in}+\frac{i\varepsilon EC_{12}}{E^2-\varepsilon^2C_{12}C_{21}}(e^{-i\tau}-1)f_2^{in}\\
&U_2(0,x,\tau)=\frac{i\varepsilon EC_{21}}{E^2-\varepsilon^2C_{12}C_{21}}(e^{-i\tau}-1)f_1^{in}+e^{-i\tau}f_2^{in},
\end{aligned}
\end{equation}
where $u_0(x)=U(0,x,0)=(f_1^{in}(x),f_2^{in}(x))^\top.$

\subsection{Quantum simulation of system \eqref{exteqngo}}

In this section, we consider the quantum algorithm to solve \eqref{exteqngo}.
Since the phase $S$ in \eqref{twodimSeq} can be solved as in section~\ref{one-dim}, we will omit the detailed computation for it and suppose it has been well approximated.
\subsubsection{Quantum simulation with the spectral method}\label{QSNGO}

By applying the Fourier spectral discretisation on $x$ and $\tau$ in \eqref{exteqngo}, we can get
\begin{subequations}
\begin{align}
&\partial_{t} \boldsymbol{u}+i(A_1 P_{x})\otimes I_{M_\tau}\boldsymbol{u}-C_{11}\boldsymbol{u}-C_{12}T_1\boldsymbol{v}=\frac{1}{\varepsilon}{\text{diag}}\left\{(A_1-A_2)P_x \boldsymbol{S}-i\boldsymbol{E}\right\}\otimes P_{\tau} \boldsymbol{u},\label{PDEdis11}\\
&\partial_{t} \boldsymbol{v}+i(A_2 P_{x})\otimes I_{M_\tau}\boldsymbol{v}-C_{21}T_2\boldsymbol{u}-C_{22}\boldsymbol{v}=-\frac{i}{\varepsilon}{\text{diag}}\left\{\boldsymbol{E}\right\}\otimes P_{\tau} \boldsymbol{v},\label{PDEdis12}\\
&\partial_t \boldsymbol{S}+iA_2 P_x \boldsymbol{S}=\boldsymbol{E},\label{PDEdis13}
\end{align}
\end{subequations}
where $\boldsymbol{u}=\sum_{j,k}U_1(t,x_j,\tau_k)\ket{j}\ket{k},\boldsymbol{v}=\sum_{j,k}V_2(t,x_j,\tau_k)\ket{j}\ket{k}$ and $\boldsymbol{S}=\sum_{j} S(t,x_j)\ket{j}$.
The matrices $A_1,A_2,\boldsymbol{E}$ and $T_1,T_2$ are defined by
$$A_1={\text {diag}}\left\{\sum\nolimits_{j} a_1(x_j)\ket{j}\right\},A_2={\text{diag}}\left\{\sum\nolimits_{j} a_2(x_j)\ket{j}\right\},\boldsymbol{E}=\sum\nolimits_{j} E(x_j)\ket{j},$$
$$T_1=I_{M_x}\otimes {\text{diag}}\left\{\sum\nolimits_{k}e^{-i\tau_k}\ket{k}\right\},T_2=I_{M_x}\otimes {\text {diag}}\left\{\sum\nolimits_{k}e^{i\tau_k}\ket{k}\right\}=T_1^\dagger.$$
Then we can further write (\ref{PDEdis11}-\ref{PDEdis12}) in vector form as $\dot{\boldsymbol{z}}=M\boldsymbol{z}$ where $\boldsymbol{z}=\left[\boldsymbol{u};\boldsymbol{v}\right]$ and
\begin{align*}
M&=\sigma_{00}\otimes(-i(A_1 P_{x})\otimes I_{M_\tau}+C_{11}I_{M_x}\otimes I_{M_\tau}+\frac{1}{\varepsilon}{\text{diag}}\left\{(A_1-A_2)P_x \boldsymbol{S}-i\boldsymbol{E}\right\}\otimes P_{\tau})\\
&+\sigma_{11}\otimes(-i(A_2 P_{x})\otimes I_{M_\tau}+C_{22}I_{M_x}\otimes I_{M_\tau}-\frac{i}{\varepsilon}{\text{diag}}\left\{\boldsymbol{E}\right\}\otimes P_{\tau})\\
&+C_{12}\sigma_{01}\otimes T_1+C_{21}\sigma_{10}\otimes T_2.
\end{align*}
Moreover, we can find
\begin{align*}
H_1&=\sigma_{00}\otimes \left({\text{Re}}(C_{11})I_{M_x}\otimes I_{M_\tau}-\frac{iA_1P_x-iP_xA_1}{2}\otimes I_{M_\tau}\right)\\
&+\sigma_{11}\otimes \left({\text{Re}}(C_{22})I_{M_x}\otimes I_{M_\tau}-\frac{iA_2P_x-iP_xA_2}{2}\otimes I_{M_\tau}\right)\\
&+\frac{C_{12}+\overline{C_{21}}}{2}\sigma_{01}\otimes T_1+\frac{\overline{C_{12}}+C_{21}}{2}\sigma_{10}\otimes T_2,
\end{align*}
and
\begin{align*}
H_2&=\sigma_{00}\otimes\left({\text{Im}}(C_{11})I_{M_x}\otimes I_{M_\tau}-\frac{A_1P_x+P_xA_1}{2}\otimes I_{M_\tau}-\frac{1}{\varepsilon}{\text{diag}}\left\{i(A_1-A_2)P_x \boldsymbol{S}+\boldsymbol{E}\right\}\otimes P_{\tau}\right)\\
&+\sigma_{11}\otimes\left({\text{Im}}(C_{22})I_{M_x}\otimes I_{M_\tau}-\frac{A_2P_x+P_xA_2}{2}\otimes I_{M_\tau}-\frac{1}{\varepsilon}{\text{diag}}\left\{\boldsymbol{E}\right\}\otimes P_{\tau}\right)\\
&+i\frac{\overline{C_{21}}-C_{12}}{2}\sigma_{01}\otimes T_1+i\frac{\overline{C_{12}}-C_{21}}{2}\sigma_{10}\otimes T_2.
\end{align*}
Then the homogeneous system can be solved by Schr{\"o}dingerisation with a well approximated $\boldsymbol{S}$.
Since the corresponding Hermitian matrix of $M$ is time-dependent, strategy in section~\ref{timedepway} can be taken.
On the other hand, Schr{\"o}dingerisation can also be used for $\boldsymbol{S}$ via \eqref{PDEdis13}.
When $A_1,A_2$ are both constant matrices, it's easy to check that 
$$\lambda(H_1)=\frac{{\text{Re}}(C_{11}+C_{22})+\sqrt{({\text{Re}}(C_{11}-C_{22}))^2+\vert C_{12}+\overline{C_{21}}\vert^2}}{2}$$ 
so that $H_1$ will be a zero matrix, i.e. $M=iH_2$ if ${\text{Re}}(C_{11})={\text{Re}}(C_{22})=C_{12}+\overline{C_{21}}=0$.
Thus, in this case the warped phase transformation is not needed for $\boldsymbol{z}$.

%
%
%

\subsubsection{Quantum simulation with the finite difference discretisation}

By denoting $U_{j,k}(t)\approx U_1(t,x_j,\tau_k)$ and $V_{j,k}\approx V_2(t,x_j,\tau_k)$, the approximations to \eqref{exteqngo} can be obtained by the following upwind numerical scheme
\begin{equation}\label{PDEFDM}
\begin{aligned}
&\frac{d U_{j,k}}{dt}+a_1(x_j)\frac{U_{j,k}-U_{j-1,k}}{\Delta x}-C_{11}U_{j,k}-C_{12}e^{-i\tau_k}V_{j,k}\\
&=-\frac{1}{\varepsilon}\left(E_j+(a_1(x_j)-a_2(x_j))\partial_x S(x_j)\right)\frac{U_{j,k}-U_{j,k-1}}{\Delta \tau},\\
&\frac{d V_{j,k}}{d t}+a_2(x_j)\frac{V_{j,k}-V_{j-1,k}}{\Delta x}-C_{21}e^{i\tau_k}U_{j,k}-C_{22}V_{j,k}\\
&=-\frac{1}{\varepsilon}E_j\frac{V_{j,k}-V_{j,k-1}}{\Delta \tau},\\
\end{aligned}
\end{equation}
where $E_j=E(t,x_j)$.
Certainly, high-order methods should be used for the approximation of $S$. 
By further denoting $\boldsymbol{u}(t)=\sum_{j,k}U_{j,k}(t)\ket{j}\ket{k},\boldsymbol{v}(t)=\sum_{j,k}V_{j,k}(t)\ket{j}\ket{k}$ and $\boldsymbol{z}=\left[\boldsymbol{u};\boldsymbol{v}\right]$, the above scheme can be written in vector form as $\dot{\boldsymbol{z}}=M\boldsymbol{z}$
where
\begin{align*}
M&=\sigma_{00}\otimes(-(A_1 D_x^{-})\otimes I_{M_\tau}+C_{11}I_{M_x}\otimes I_{M_\tau}-\frac{1}{\varepsilon}{\text {diag}}\left\{\boldsymbol{E}+(A_1-A_2)\partial_x\boldsymbol{S}\right\}\otimes D_\tau^{-})\\
&+\sigma_{11}\otimes(-(A_2 D_x^{-})\otimes I_{M_\tau}+C_{22}I_{M_x}\otimes I_{M_\tau}-\frac{1}{\varepsilon}{\text {diag}}\left\{\boldsymbol{E}\right\}\otimes D_\tau^{-})\\
&+C_{12}\sigma_{01}\otimes T_1+C_{21}\sigma_{10}\otimes T_2.
\end{align*}
Meanwhile, we can get
\begin{align*}
H_1&=\sigma_{00}\otimes\left(\left(\frac{D_x^{+}A_1-A_1 D_x^{-}}{2}\right)\otimes I_{M_\tau}+C_{11}I_{M_x}\otimes I_{M_\tau}+\frac{1}{\varepsilon}{\text {diag}}\left\{\boldsymbol{E}+(A_1-A_2)\partial_x\boldsymbol{S}\right\}\otimes \frac{D_\tau^{+}-D_\tau^{-}}{2}\right)\\
&+\sigma_{11}\otimes\left(\left(\frac{D_x^{+}A_2-A_2 D_x^{-}}{2}\right)\otimes I_{M_\tau}+C_{22}I_{M_x}\otimes I_{M_\tau}+\frac{1}{\varepsilon}{\text {diag}}\left\{\boldsymbol{E}\right\}\otimes \frac{D_\tau^{+}-D_\tau^{-}}{2}\right)\\
&+\frac{C_{12}+C_{21}}{2}\sigma_{01}\otimes T_1+\frac{C_{12}+C_{21}}{2}\sigma_{10}\otimes T_2,
\end{align*}
and
\begin{align*}
H_2&=i\sigma_{00}\otimes\left(\left(\frac{D_x^{+}A_1+A_1 D_x^{-}}{2}\right)\otimes I_{M_\tau}+\frac{1}{\varepsilon}{\text {diag}}\left\{\boldsymbol{E}+(A_1-A_2)\partial_x\boldsymbol{S}\right\}\otimes \frac{D_\tau^{+}+D_\tau^{-}}{2}\right)\\
&+i\sigma_{11}\otimes\left(\left(\frac{D_x^{+}A_2+A_2 D_x^{-}}{2}\right)\otimes I_{M_\tau}+\frac{1}{\varepsilon}{\text {diag}}\left\{\boldsymbol{E}\right\}\otimes \frac{D_\tau^{+}+D_\tau^{-}}{2}\right)\\
&+i\frac{C_{21}-C_{12}}{2}\sigma_{01}\otimes T_1+i\frac{C_{12}-C_{21}}{2}\sigma_{10}\otimes T_2.
\end{align*}

\subsection{Numerical results}

In this section, we will present the numerical experiments for the new approach.
The parameters in \eqref{exteqs} are chosen as $a_1(x)=1,a_2(x)=4,E(t,x)=1.5+\cos(x)$, and
\begin{equation*}
C=\left[
            \begin{array}{cc}
              0 & 1 \\
              -1 & 0 
            \end{array}
          \right].
\end{equation*}
The initial condition is
$$u(0,x)=\left(1+\frac{1}{2}\cos(x)+i\sin(x),1+\frac{1}{2}\cos(x)+i\sin(x)\right),\quad x\in[0,2\pi].$$

In the following figures, we compare a reference solution (computed with time splitting method with resolved numerical computation) and the solution of our quantum algorithm with the corrected initial condition \eqref{NGOinitial}. 
For the reference solution, we set the cell number as $M_x=2^8$.

In Figure~\ref{imgeg31} , we plot the real and imaginary part of components of the solution as a function of space for final time $T=1$. 
We can observe that the our method is able to capture the high space oscillations of the solution with a spatial mesh independent of $\varepsilon$.

\begin{figure}[!ht]
\centering
\subfigure[]{
\includegraphics[scale=.3]{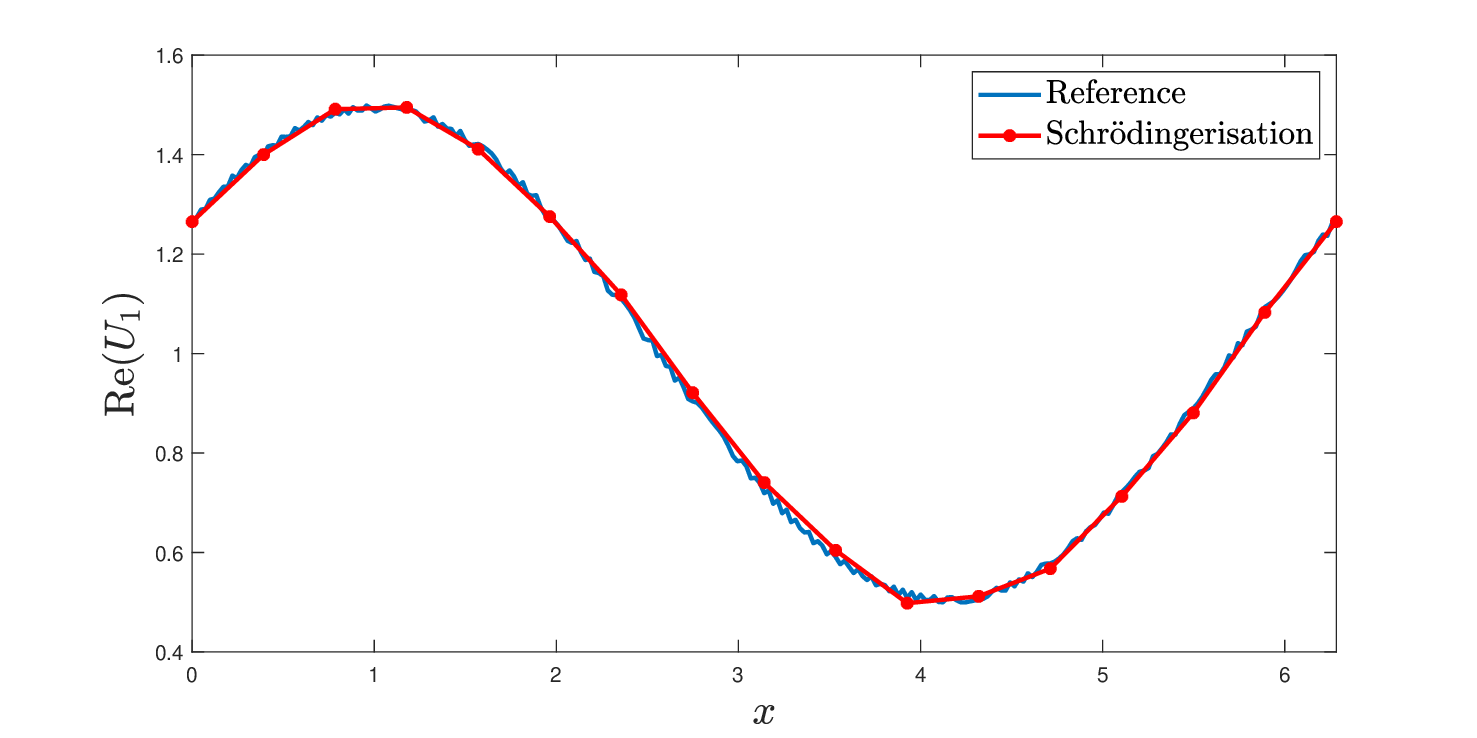}
}
\subfigure[]{
\includegraphics[scale=.3]{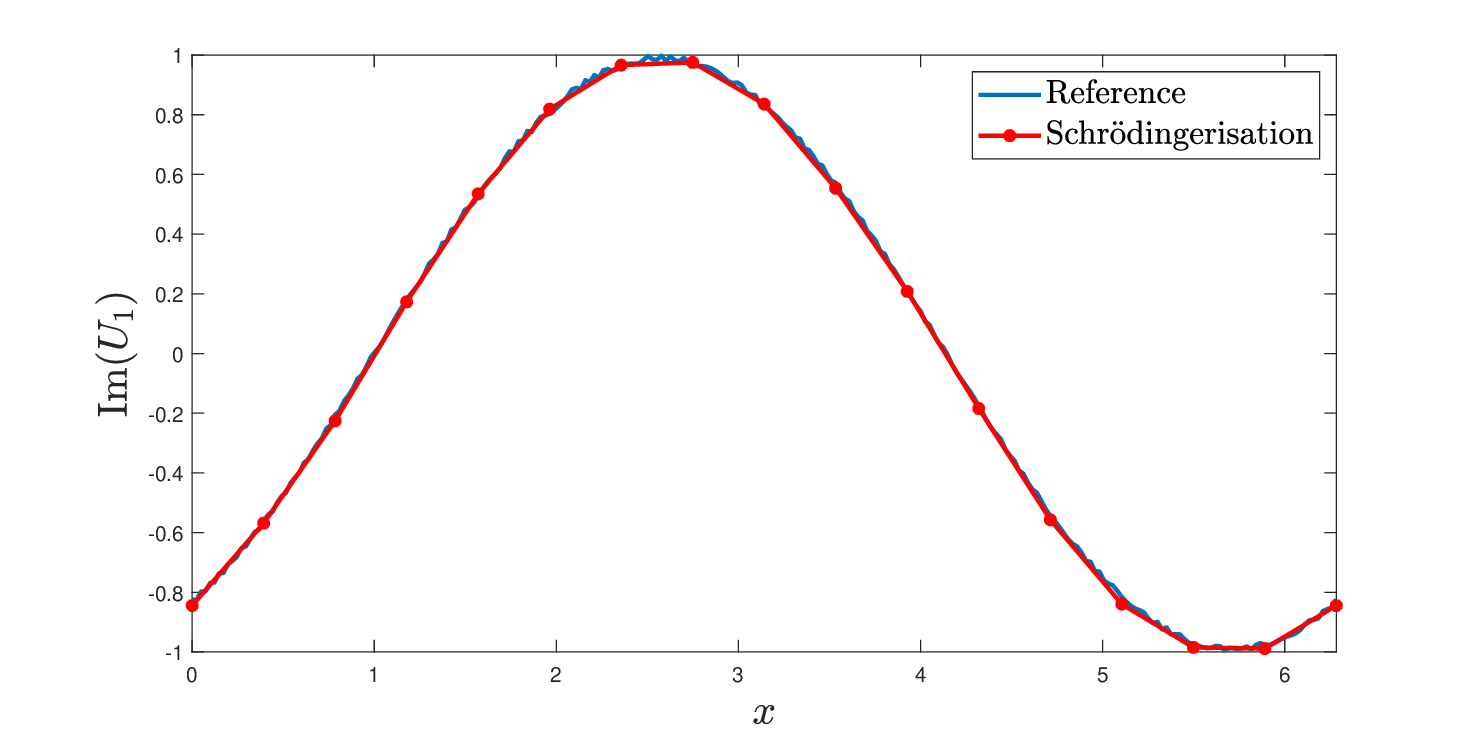}
}
\subfigure[]{
\includegraphics[scale=.3]{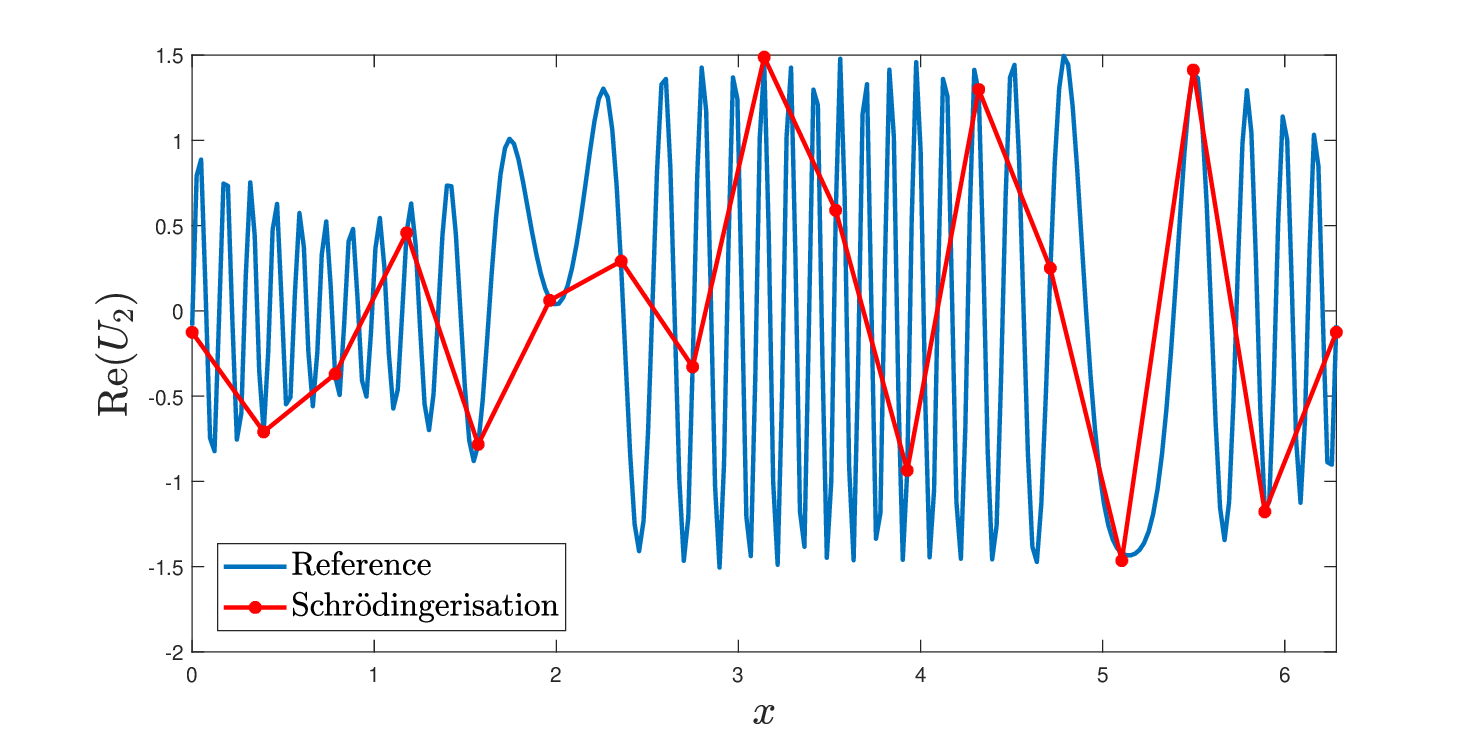}
}
\subfigure[]{
\includegraphics[scale=.3]{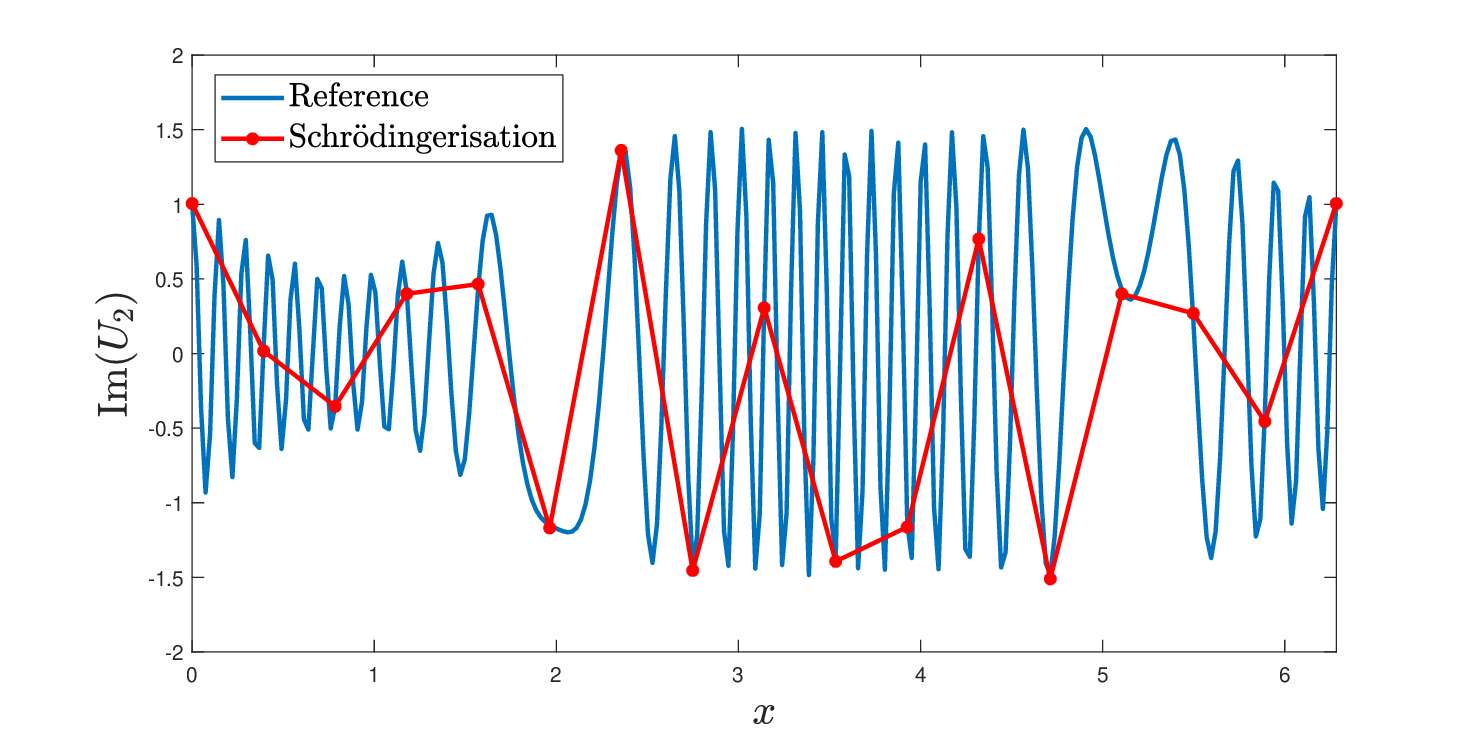}
}
\caption{Comparison between the reference solution and the numerical solution (with exact solution for $S$) for $\varepsilon=0.01,T=1,l=3,m=4$. (a), (b): $U_1$.
(c), (d): $U_2$.}
\label{imgeg31}
\end{figure}

Then, we illustrate the performances of the Schr{\"o}dingerisation method by using the same data as before except the constant matrix
\begin{equation*}
C=\frac{1}{2}\left[
            \begin{array}{cc}
              1+i & 1+i \\
              -1+i & 1-i 
            \end{array}
          \right].
\end{equation*}
The computational domain of $p$ is $[-5,5]$.
According to Corollary~\ref{lambdachoice}, we choose $\lambda_0=\frac{T}{2}$ and non-negative $p_k$ to recover the solution $\boldsymbol{z}$.
As shown in Figure~\ref{imgeg32}, the Schr{\"o}dingerisation method captures the high space oscillations of the solution, point-wise,  well even if the spatial mesh $\Delta{x}=\frac{2\pi}{2^4}$ is large compared to the size of the oscillation ($\varepsilon=0.01$).

\begin{figure}[!ht]
\centering
\subfigure[]{
\includegraphics[scale=.3]{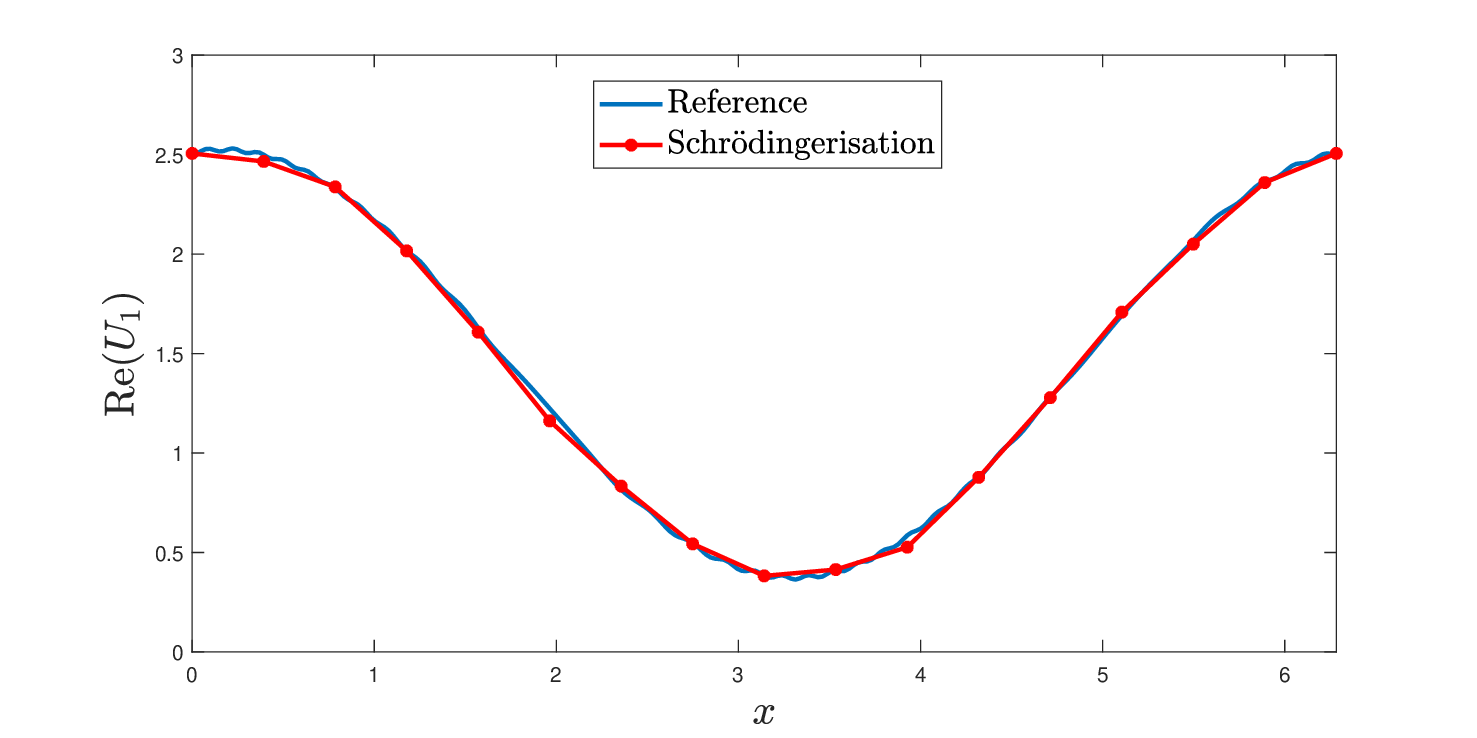}
}
\subfigure[]{
\includegraphics[scale=.3]{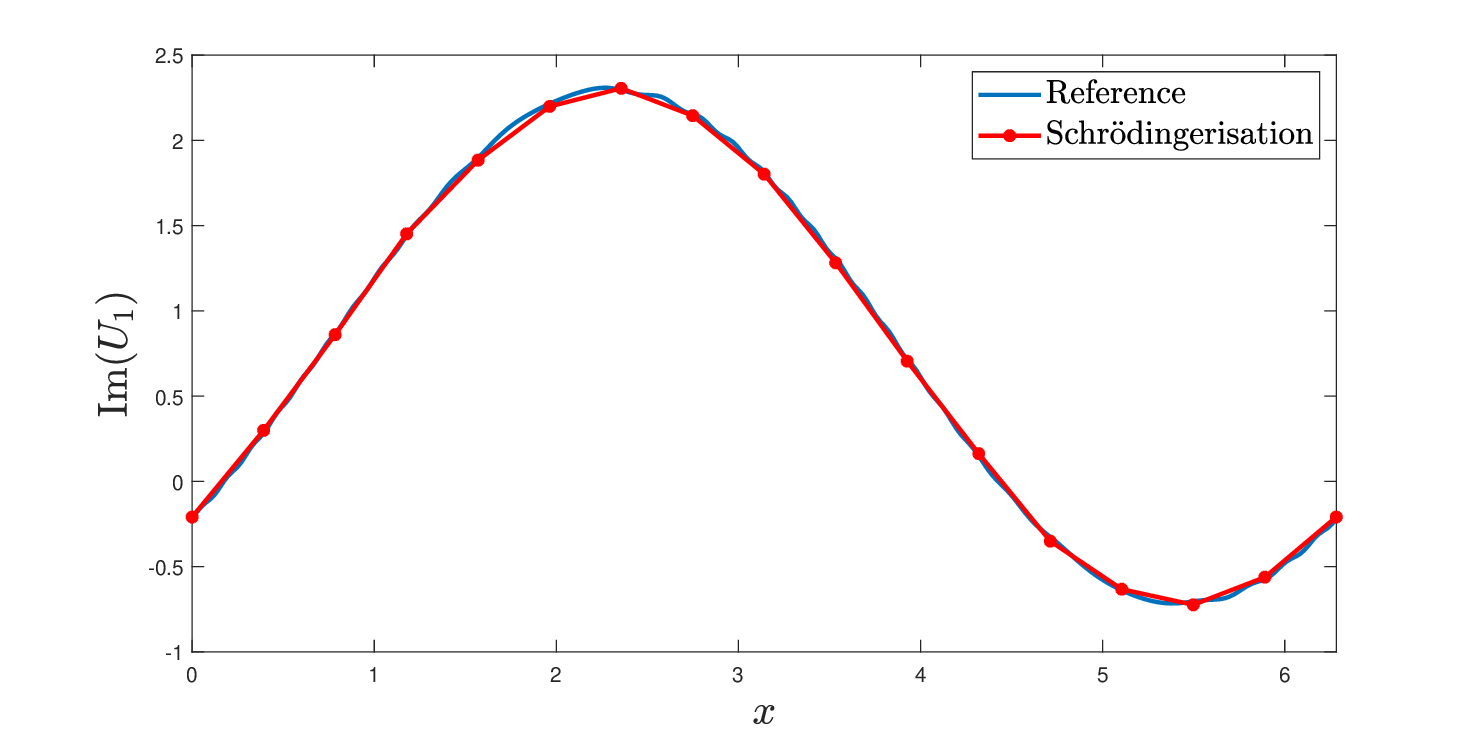}
}
\subfigure[]{
\includegraphics[scale=.3]{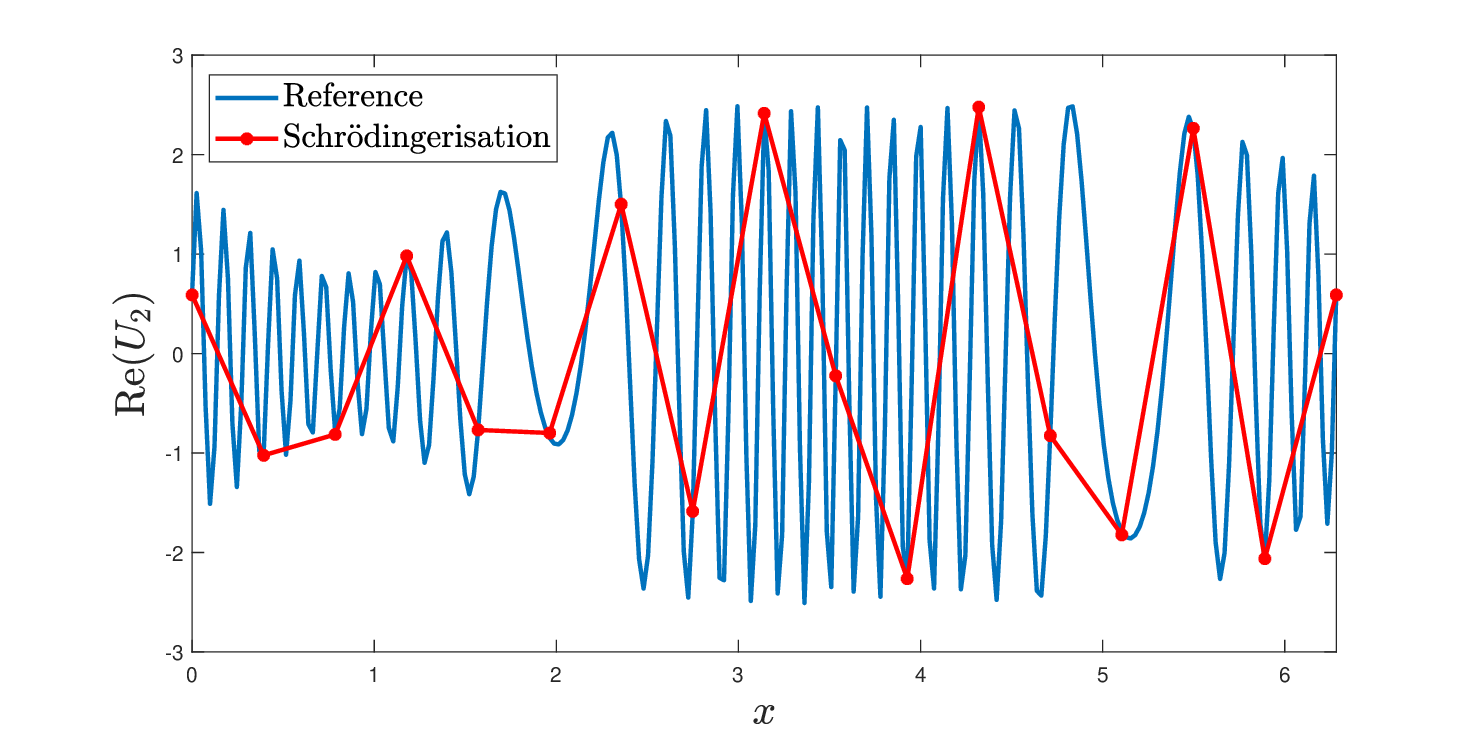}
}
\subfigure[]{
\includegraphics[scale=.3]{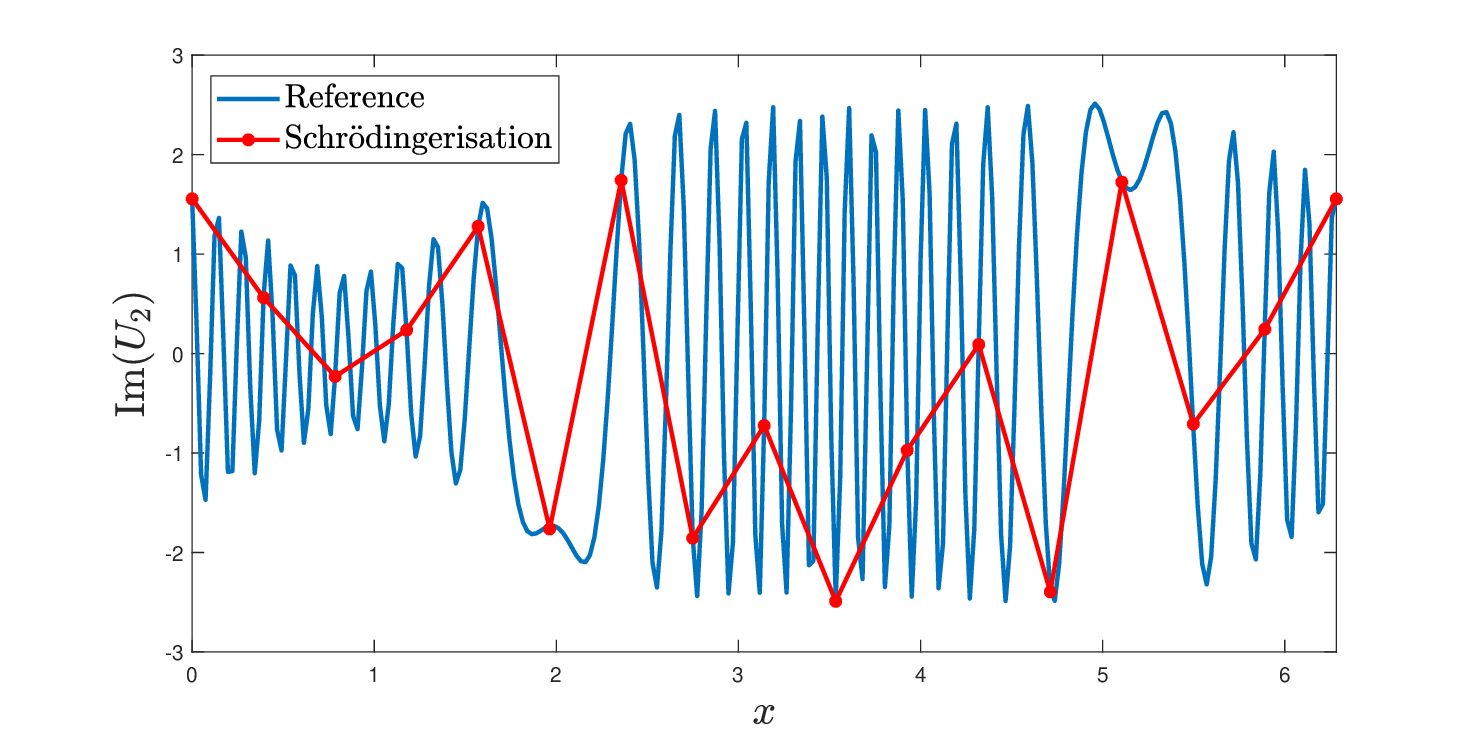}
}
\caption{Comparison between the reference solution and the numerical solution of Schr{\"o}dingerisation (with exact solution for $S$) for $\varepsilon=0.01,T=1,l=3,m=4,n=8$. (a), (b): $U_1$.
(c), (d): $U_2$.}
\label{imgeg32}
\end{figure}

\subsection{Quantum simulation for the surface hopping model}
Finally, we show that the general quantum approach described above can be extended to efficiently solve the following semi-classical surface hopping model.
This model approximates semiclassically the nucleaonic Schr{\"o}dinger system arising from the Born-Oppenheimer approximation with non-adiabatic corrections. For more details see~\cite{Chai2015}.
Specifically, this model can be expressed as the following equations,
\begin{equation}\label{hopping}
\begin{aligned}
&\partial_t f^{+}+v\cdot\nabla_x f^{+}-\nabla_x(U+E)\cdot\nabla_v f^{+}=\overline{b^{i}}f^{i}+b^{i}\overline{f^{i}},\\
&\partial_t f^{-}+v\cdot\nabla_x f^{-}-\nabla_x(U-E)\cdot\nabla_v f^{-}=-\overline{b^{i}}f^{i}-b^{i}\overline{f^{i}},\\
&\partial_t f^{i}+v\cdot\nabla_x f^{i}-\nabla_x U\cdot\nabla_v f^{i}=-i\frac{2E}{\varepsilon}f^{i}+b^{i}(f^{-}-f^{+})+(b^{+}-b^{-})f^{i},
\end{aligned}
\end{equation}
where $(f^{+}(t,x,v),f^{-}(t,x,v),f^{i}(t,x,v))\in\mathbb{R}\times\mathbb{R}\times\mathbb{C},(t,x,v)\in\mathbb{R}_{+}\times\mathbb{R}^d\times\mathbb{R}^d$, and $b^{\pm},b^{i}\in\mathbb{C},U,E\in\mathbb{R}$ are given functions depending only on the space variable $x$.
The initial conditions are denoted by 
$$(f^{+}(0,x,v),f^{-}(0,x,v),f^{i}(0,x,v))=(f^{+}_{in}(x,v),f^{-}_{in}(x,v),f^{i}_{in}(x,v)).$$
In vector form, \eqref{hopping} becomes
\begin{equation}\label{hoppingvec}
\partial_t \boldsymbol{f}+\boldsymbol{v}\cdot\nabla_x \boldsymbol{f}-\nabla_x A\cdot\nabla_v \boldsymbol{f}=C\boldsymbol{f},
\end{equation}
where
$$\boldsymbol{f}=(f^{+},f^{-},\text{Re}(f^{i}),\text{Im}(f^{i}))^\top,\ \boldsymbol{v}=(v,v,v,v)^\top,\ A={\text{diag}}(U+E,U-E,U,U),$$
$$C=\left[
            \begin{array}{cccc}
              0\quad & 0\quad & b^{i}+\overline{b^{i}}\quad & -ib^{i}+i\overline{b^{i}}\\
              0\quad & 0\quad & -b^{i}-\overline{b^{i}}\quad & ib^{i}-i\overline{b^{i}}\\
             -b^{i}\quad & b^{i}\quad & b^{+}-b^{-}\quad & 2E/\varepsilon\\
             0\quad & 0\quad & -2E/\varepsilon\quad & b^{+}-b^{-}\\
            \end{array}
          \right].$$
A direct time splitting scheme for \eqref{hoppingvec} writes 
\begin{itemize}
    \item solve $\partial_t \boldsymbol{f}+\boldsymbol{v}\cdot\nabla_x \boldsymbol{f}=0$ by spectral method in space and exact integration in time,
    \item solve $\partial_t \boldsymbol{f}-\nabla_x A\cdot\nabla_v \boldsymbol{f}=0$ by spectral method in velocity and exact integration in time,
    \item solve $\partial_t \boldsymbol{f}=C\boldsymbol{f}$ by exact integration.
\end{itemize}

For an explanation of quantum simulation on \eqref{hoppingvec}, we consider $b^{i}\in\mathbb{R},b^{\pm}=0$ which corresponds to two potential matrices introduced in~\cite{Chai2015}.
By introducing a phase $S(t,x,v)$, solution to
\begin{equation}\label{hoppingS}
\partial_t S+v\cdot\nabla_x S-\nabla_x U\cdot\nabla_v S=2E,\ S(0,x,v)=0,
\end{equation}
and the augmented unknowns $(F^{\pm},F^{i})(t,x,v,\tau)$ with $\tau=S(t,x,v)/\varepsilon$ satisfying
$$f^{\pm}(t,x,v)=F^{\pm}(t,x,v,\tau),\ f^{i}(t,x,v)=e^{-i\tau}(G+iH)(t,x,v,\tau),$$
where $G=\text{Re}(e^{i\tau}f^{i}),H=\text{Im}(e^{i\tau}f^{i})$.
Then it has
\begin{equation}\label{hoppingF}
\begin{aligned}
&\partial_t F^{+}+v\cdot\nabla_x F^{+}-\nabla_x(U+E)\cdot\nabla_v F^{+}=-\frac{\mathcal{E}^{+}}{\varepsilon}\partial_{\tau}F^{+}+2b^{i}G\cos{\tau}+2b^{i}H\sin{\tau},\\
&\partial_t F^{-}+v\cdot\nabla_x F^{-}-\nabla_x(U-E)\cdot\nabla_v F^{-}=-\frac{\mathcal{E}^{-}}{\varepsilon}\partial_{\tau}F^{-}-2b^{i}G\cos{\tau}-2b^{i}H\sin{\tau},\\
&\partial_t G+v\cdot\nabla_x G-\nabla_x U\cdot\nabla_v G=-\frac{2E}{\varepsilon}\partial_{\tau}G+b^{i}(F^{-}-F^{+})\cos{\tau},\\
&\partial_t H+v\cdot\nabla_x H-\nabla_x U\cdot\nabla_v H=-\frac{2E}{\varepsilon}\partial_{\tau}H+b^{i}(F^{-}-F^{+})\sin{\tau},
\end{aligned}
\end{equation}
where $\mathcal{E}^{\pm}=2E\mp\nabla_x E\cdot\nabla_v S$.
The initial data can be derived from the Chapman-Enskog expansion which yields
\begin{equation}\label{hoppinginitial}
\begin{aligned}
&F^{+}(0,x,v,\tau)=f^{+}_{in}-i\frac{\varepsilon}{2E}\left(b^{i}f^{i}_{in}(1-e^{-i\tau})-b^{i}\overline{f^{i}_{in}}(1-e^{i\tau})\right),\\
&F^{-}(0,x,v,\tau)=f^{-}_{in}+i\frac{\varepsilon}{2E}\left(b^{i}f^{i}_{in}(1-e^{-i\tau})-b^{i}\overline{f^{i}_{in}}(1-e^{i\tau})\right),\\
&G(0,x,v,\tau)=\text{Re}(f^{i}_{in})-\frac{\varepsilon}{2E}b^{i}(f^{+}_{in}-f^{-}_{in})\sin{\tau},\\
&H(0,x,v,\tau)=\text{Im}(f^{i}_{in})+\frac{\varepsilon}{2E}b^{i}(f^{+}_{in}-f^{-}_{in})(\cos{\tau}-1).
\end{aligned}
\end{equation}

Then we apply the Schr{\"o}dingerisation method for \eqref{hoppingF} with initial data \eqref{hoppinginitial}.
Without loss of generality, we set $d=1$.
By denoting $\boldsymbol{g}=(F^{+},F^{-},G,H)^\top$, \eqref{hoppingF} becomes the following vector form
\begin{equation}\label{hoppingvector2}
\partial_t \boldsymbol{g}+v\partial_x \boldsymbol{g}+A\partial_v \boldsymbol{g}+B\partial_{\tau}\boldsymbol{g}=C\boldsymbol{g},
\end{equation}
where
$$A=\partial_x {\text{diag}} (-U-E,-U+E,-U,-U),\ B={\text{diag}} \left(\frac{\mathcal{E}^{+}}{\varepsilon},\frac{\mathcal{E}^{-}}{\varepsilon},\frac{2E}{\varepsilon},\frac{2E}{\varepsilon}\right),$$
$$C=\left[
            \begin{array}{cccc}
              0\quad & 0\quad & 2b^{i}\cos{\tau}\quad & 2b^{i}\sin{\tau}\\
              0\quad & 0\quad & -2b^{i}\cos{\tau}\quad & -2b^{i}\sin{\tau}\\
              -b^{i}\cos{\tau}\quad & b^{i}\cos{\tau}\quad & 0\quad & 0\\
              -b^{i}\sin{\tau}\quad & b^{i}\sin{\tau}\quad & 0\quad & 0
            \end{array}
          \right].$$
By applying the Fourier spectral discretisation on $x,v,\tau$ in \eqref{hoppingvector2}, we can get
\begin{equation}\label{hoppingODE}
\begin{aligned}
&\frac{d}{dt} \tilde{\boldsymbol{g}}=(-i(P_{x}\otimes I_4\otimes {\text {diag}}\{\sum_j v_j\ket{j}\}\otimes I_{M_\tau}+\sum\nolimits_{i} \ket{i}\bra{i}\otimes A(x_i)\otimes P_{v}\otimes I_{M_\tau}\\
&+\sum\nolimits_{i,j} \ket{i}\bra{i}\otimes B(t,x_i,v_j)\otimes I_{M_v}\otimes P_{\tau})+\sum\nolimits_{i,j,k}\ket{i}\bra{i}\otimes\ket{j}\bra{j}\otimes\ket{k}\bra{k}\otimes C(x_i,v_j,\tau_k))\tilde{\boldsymbol{g}},
\end{aligned}
\end{equation}
where $\tilde{\boldsymbol{g}}=\sum_{i,j,k}\boldsymbol{g}(t,x_i,v_j,\tau_k)\ket{i}\ket{j}\ket{k}$.
One can further get
\begin{equation*}
H_1=\sum\nolimits_{i,j,k}\ket{i}\bra{i}\otimes\ket{j}\bra{j}\otimes\ket{k}\bra{k}\otimes \frac{C(x_i,v_j,\tau_k)+C^{T}(x_i,v_j,\tau_k)}{2}
\end{equation*}
and
\begin{align*}
H_2&=-(P_{x}\otimes I_4\otimes {\text {diag}}\{\sum\nolimits_j v_j\ket{j}\}\otimes I_{M_\tau}+\sum\nolimits_{i} \ket{i}\bra{i}\otimes A(x_i)\otimes P_{\tau}\otimes I_{M_\tau}\\
&+\sum\nolimits_{i} \ket{i}\bra{i}\otimes B(x_i)\otimes I_{M_v}\otimes P_{\tau})\\
&-i\sum\nolimits_{i,j,k}\ket{i}\bra{i}\otimes\ket{j}\bra{j}\otimes\ket{k}\bra{k}\otimes \frac{C(x_i,v_j,\tau_k)-C^{T}(x_i,v_j,\tau_k)}{2}.
\end{align*}
And one can further get $\lambda_n(H_1)=\frac{\Vert b^{i}\Vert_{\infty}}{\sqrt{2}}$.

\subsection{Numerical results}

This section is devoted to the numerical experiments of the Schr{\"o}dingerisation approach for the surface hopping model.
We consider the following initial conditions with $x,v\in[-2\pi,2\pi]$
\begin{align*}
&f^{+}(0,x,v)=f^{-}(0,x,v)=\left(1+\frac{1}{2}\cos(x)\right)\frac{e^{-v^2/2}}{\sqrt{2\pi}},\\
&f^{i}(0,x,v)=f^{-}(0,x,v)=\left(\left(1+\frac{1}{2}\sin(x)\right)+i\left(1+\frac{1}{2}\cos(x)\right)\right)\frac{e^{-v^2/2}}{\sqrt{2\pi}},
\end{align*}
and functions
$$U(x)=0,\ E(x)=1-\cos(x/2)+\varepsilon,\ b^{i}(x,v)=-\frac{1}{2}\sin(v+1),\ b^{\pm}=0.$$
For all the following tests, the reference solution is computed by the time splitting method.
Moreover, we use notations $M_{\tau}=2^k,M_v=2^l,M_x=2^m,M_p=2^n$ for the cell number of the Schr{\"o}dingerisation method in this section.
Since we don't perform the experiments on a quantum computer now, the huge cost for solving \eqref{hoppingODE} directly by Schr{\"o}dingerisation in classic computer is unacceptable.
Thus, in the practical simulation we consider a further splitting for \eqref{hoppingODE}.
Furthermore, we choose $[-5,5]$ for the computational region of $p$.

In Figure~\ref{imgeg41}, we plot the solutions $f^{\pm,i}(T,x,v)$ and densities $\rho^{\pm,i}(T,x,v)=\int_{-2\pi}^{2\pi}f^{\pm,i}(T,x,v)dp$ for two methods.
The reference solution is computed by using $\Delta t=0.005,M_v=2^4,M_x=2^6$.
For $\varepsilon=1$, the densities evolve without oscillations.
And it can be observed that our method matches well, point-wise, the solutions and densities.

\begin{figure}[!ht]
\centering
\subfigure[]{
\includegraphics[scale=.45]{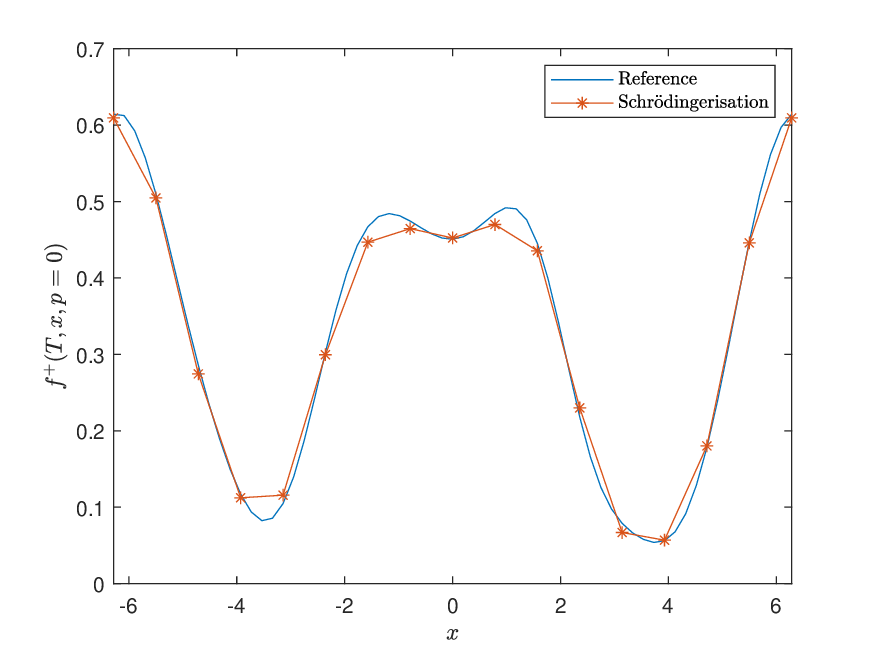}
}
\subfigure[]{
\includegraphics[scale=.45]{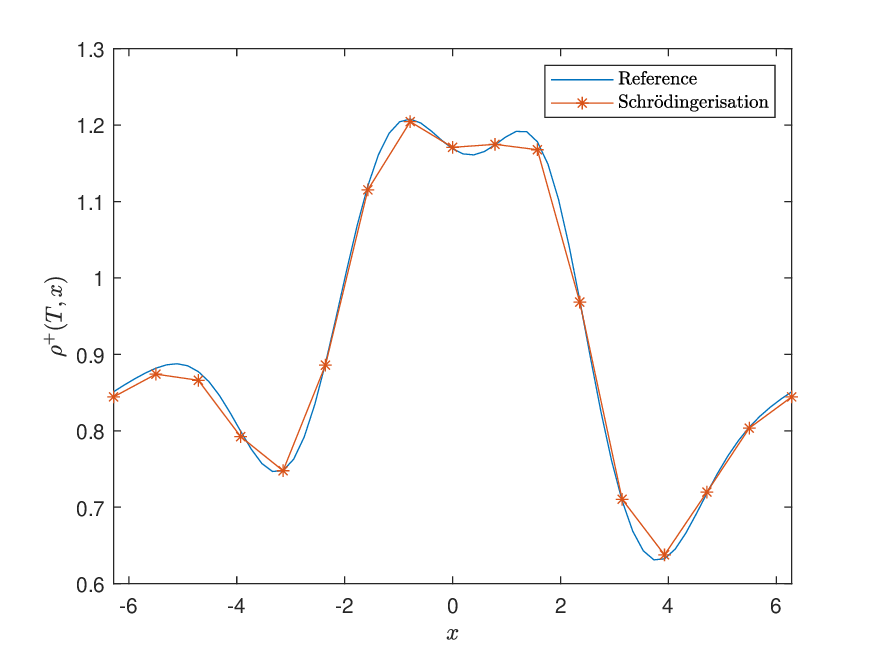}
}
\subfigure[]{
\includegraphics[scale=.45]{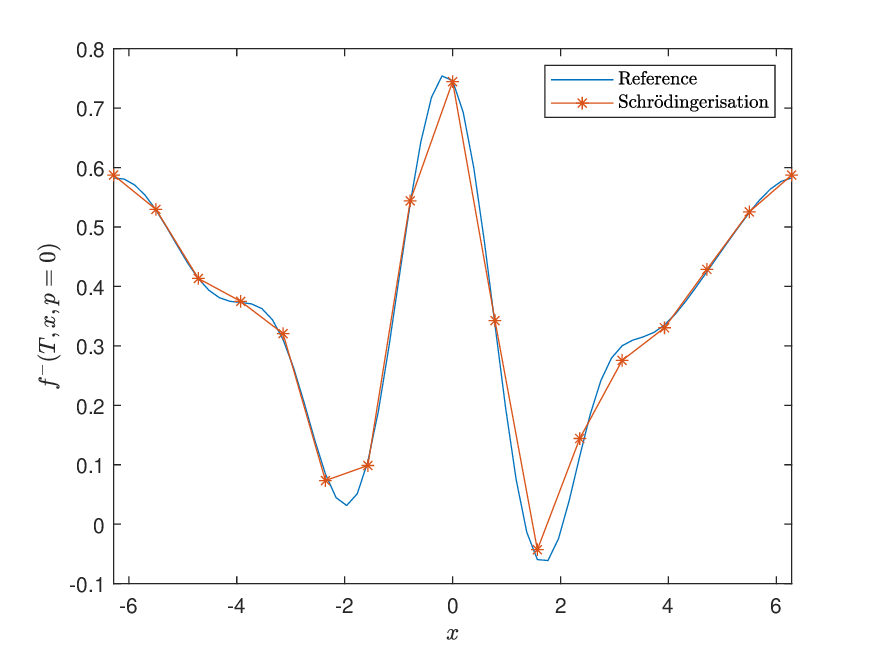}
}
\subfigure[]{
\includegraphics[scale=.45]{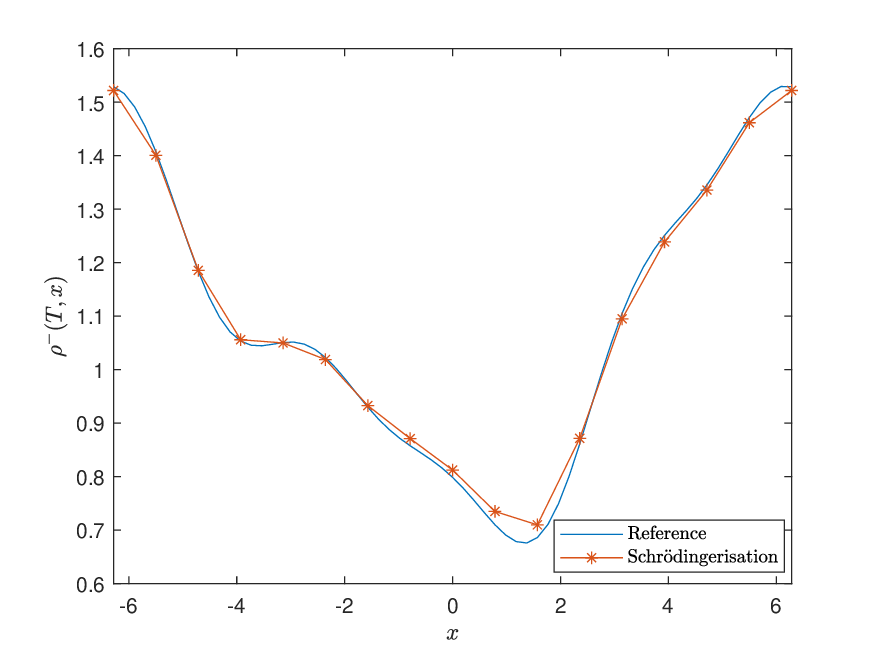}
}
\subfigure[]{
\includegraphics[scale=.45]{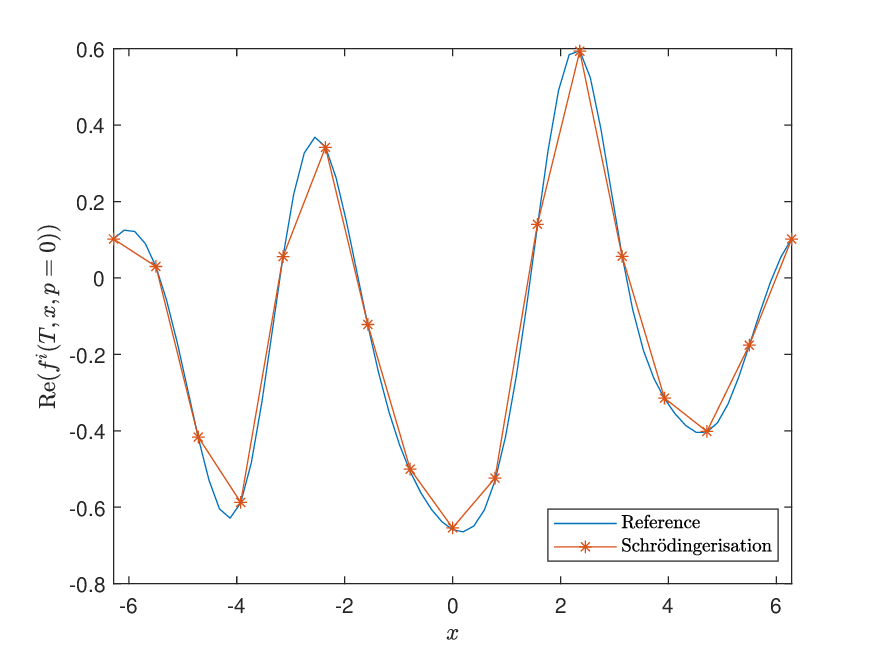}
}
\subfigure[]{
\includegraphics[scale=.45]{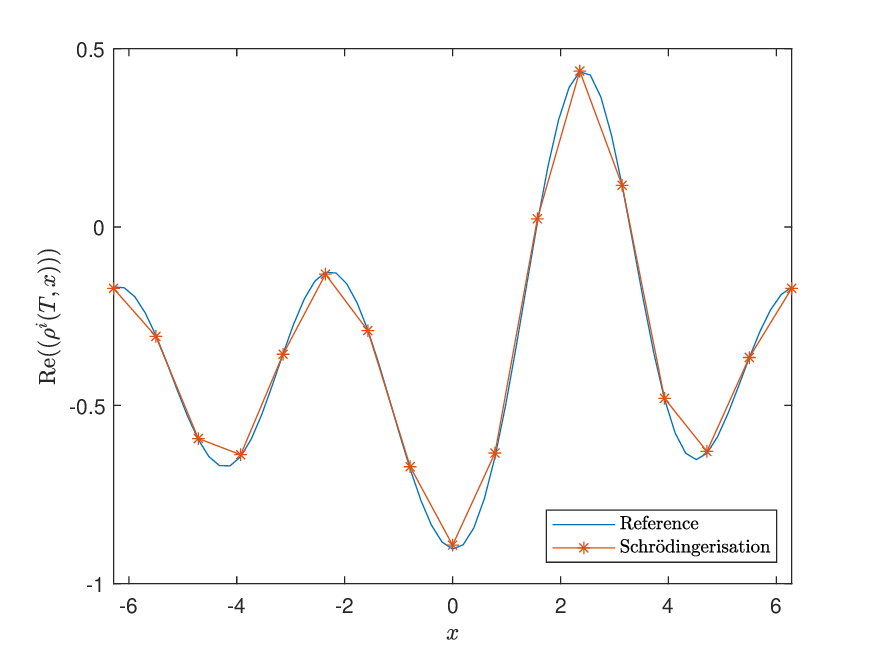}
}
\subfigure[]{
\includegraphics[scale=.45]{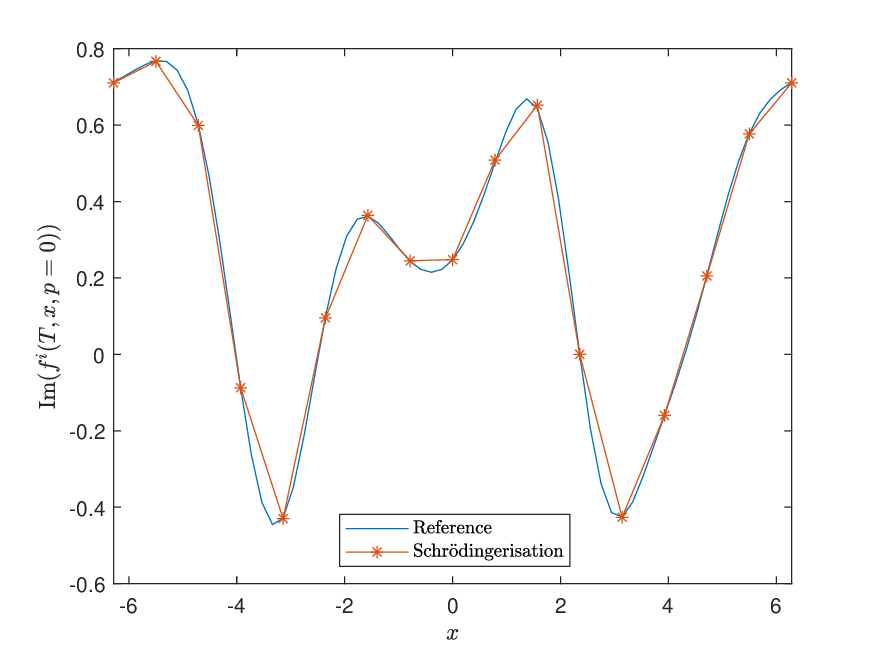}
}
\subfigure[]{
\includegraphics[scale=.45]{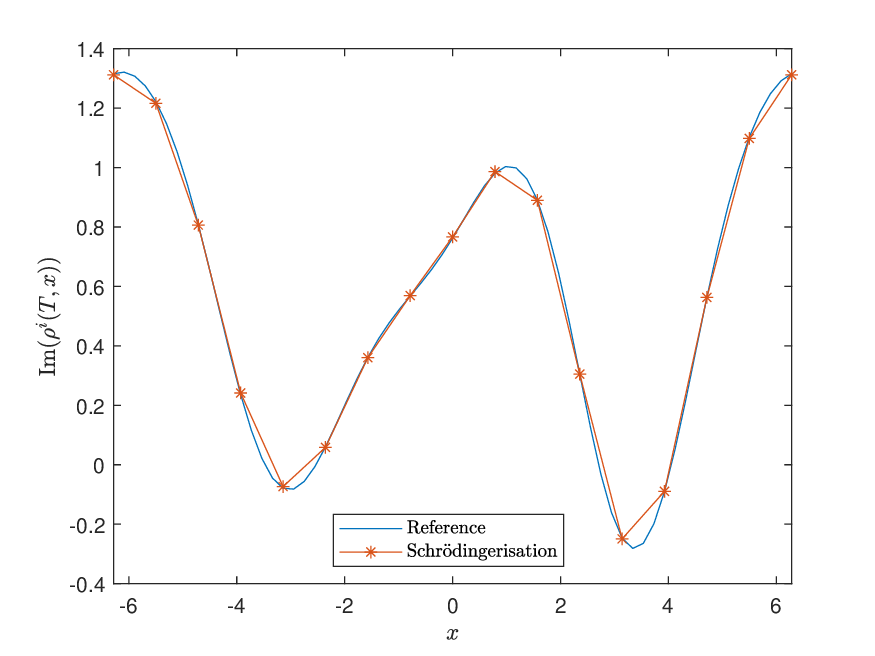}
}
\caption{Comparison between a reference solution and the numerical solution of Schr{\"o}dingerisation (with exact computation for $S$) for $\varepsilon=1,T=2,k=3,l=4,m=4,n=7$.}
\label{imgeg41}
\end{figure}

In Figure~\ref{imgeg42}, we show the densities $\rho^{\pm,i}(T,x,v)$ for two methods.
The reference solution is calculated by using $\Delta{t}=10^{-5},M_v=2^6,M_x=2^8$.
Numerical solution is computed by Schr{\"o}dingerisation with the corrected initial data \eqref{hoppinginitial} and recovered by using $\lambda_0=\frac{\sqrt{2}}{4}$ and non-negative $p_k$ according to Corollary~\ref{lambdachoice}.
In this instance, the densities are highly oscillatory due to $\varepsilon\ll 1$.
However, our method captures very well, point-wise, all densities, even though the mesh is much coarser than the spatial oscillations.

\begin{figure}[!ht]
\centering
\subfigure[]{
\includegraphics[scale=.5]{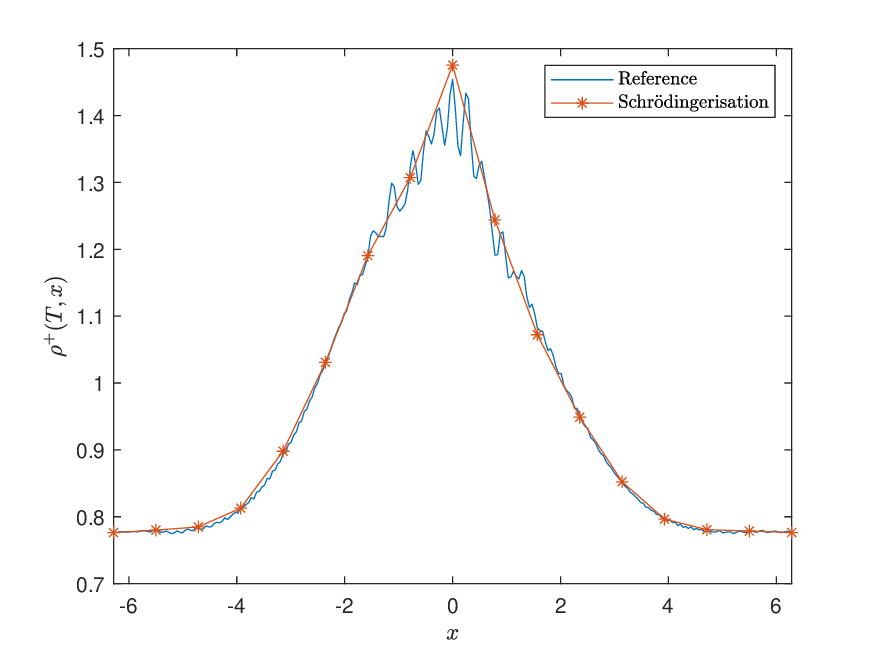}
}
\subfigure[]{
\includegraphics[scale=.5]{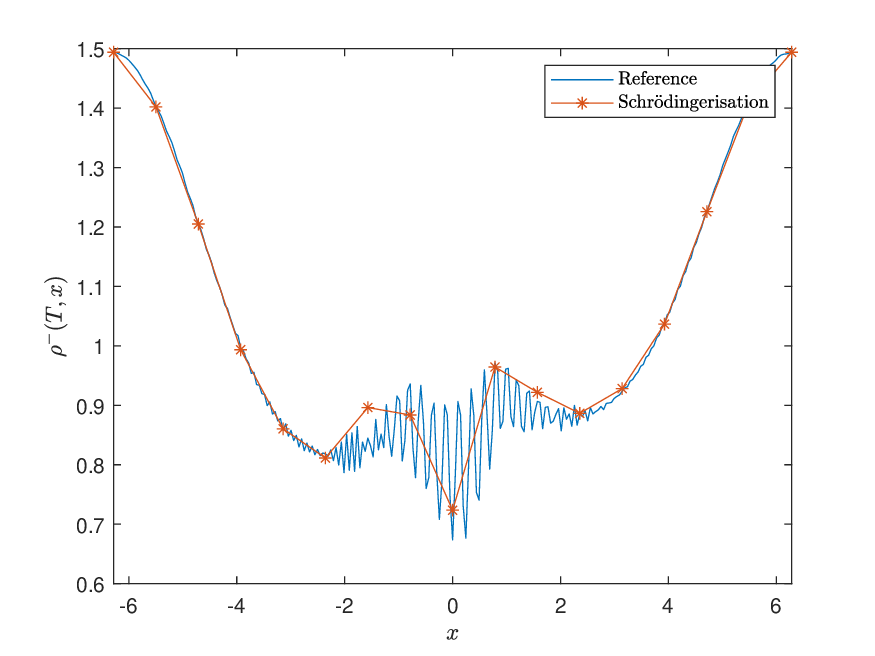}
}
\subfigure[]{
\includegraphics[scale=.5]{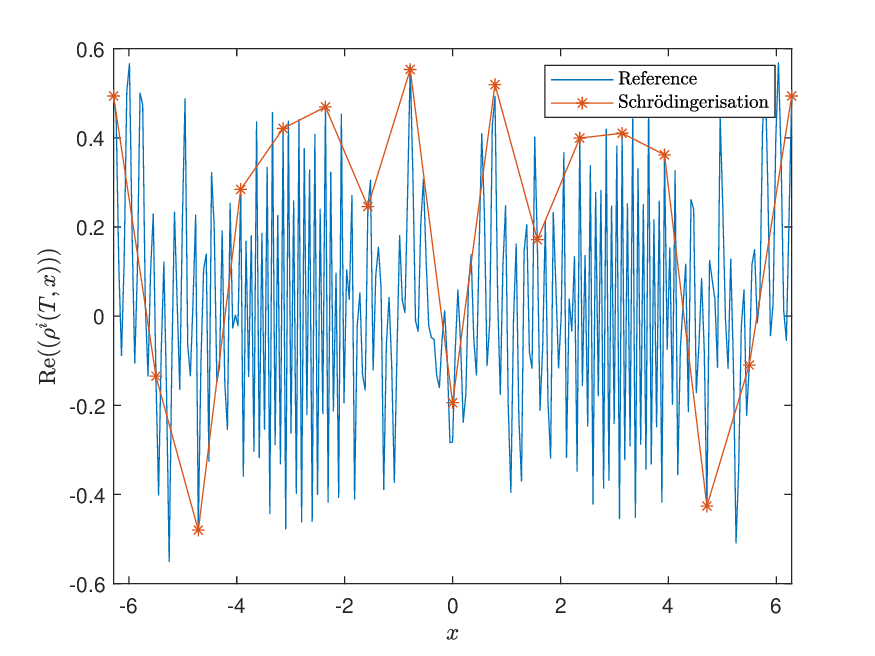}
}
\subfigure[]{
\includegraphics[scale=.5]{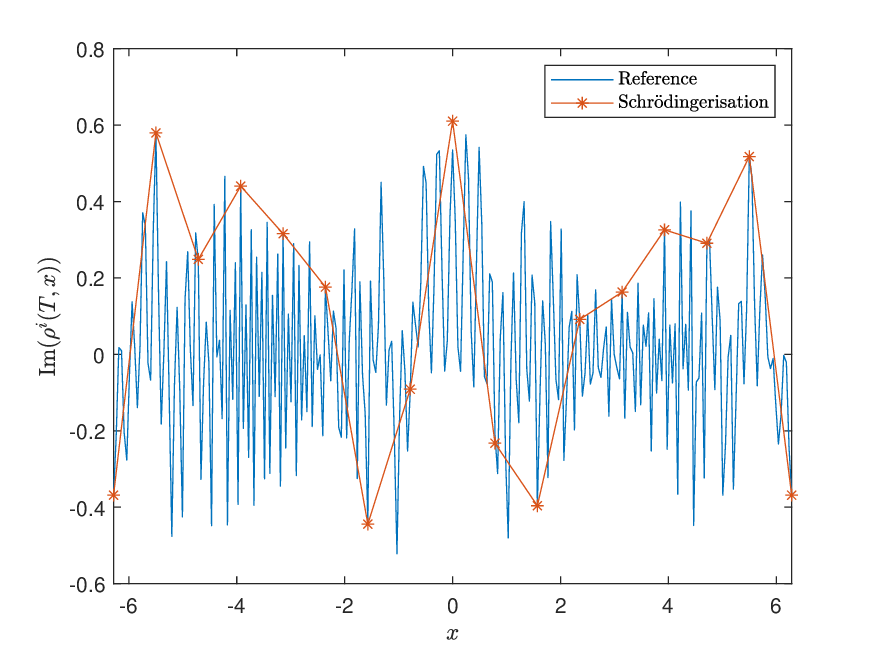}
}
\caption{Comparison between a reference solution and the solution of Schr{\"o}dingerisation (with exact computation for $S$) for $\varepsilon=1/32,T=2,k=3,l=6,m=4,n=8$.}
\label{imgeg42}
\end{figure}

\section{Conclusion}

In this paper, we proposed quantum algorithms for the following highly-oscillatory transport equations 
\begin{equation*}
\partial_t u+\sum_{k=1}^d A_k(x)\partial_{x_k} u=\frac{iE(t,x)}{\varepsilon}Du+Cu
\end{equation*}
by using the Schr{\"o}dingerisation approach introduced in~\cite{Jin2022quantum,Jin2023Quantum} combined with the nonlinear geometric optics (NGO) method~\cite{Nicolas2017Nonlinear}. 
About the Schr{\"o}dingerisation method, we also illustrate the influence of shifting the eigenvalues of linear systems' Hermite parts depending on whether they contain unstable modes or not.
This technique can help the recoveries of the original problem, and will be further studied in our following work. 

Such numerical methods allow us to obtain accurate numerical solutions, in maximum norm,  with mesh size {\it independent} of the wave length.
Our approach can be further extended to higher dimensional problems which is no longer suitable for classical computing.
This is also why we need to design quantum algorithms for these equations.
Numerical examples demonstrate that the proposed method has the desired property of capturing the pointwise solutions of highly oscillatory waves with mesh sizes independent of the wave length.

In our future work, we will design detailed quantum circuits of this method, which is important for future qubit-based general purpose quantum computers.

\section*{Acknowledgement}

The research results of this article are sponsored by the Kunshan Municipal Government research funding.



\bibliographystyle{plain}
\bibliography{refs}

\end{document}